\documentclass{mod}
\usepackage{url}
\usepackage{setspace}
\usepackage{scrextend}
\usepackage{cite}
\usepackage[shortlabels]{enumitem}
\usepackage{hyperref}
\usepackage{caption}
\usepackage{subcaption}
\usepackage{float}
\usepackage{tikz}
\usepackage[all]{xy}
\usepackage{tikz-cd}
\usepackage{amsmath}

\numberwithin{equation}{section}
\renewcommand{\theequation}{\thesection.\arabic{equation}}

\usepackage{mathtools}
\usepackage{amsthm}
\usepackage{amssymb}
\usepackage[dvipsnames]{xcolor}

\DeclareMathAlphabet{\mathbbm}{U}{bbm}{m}{n}

\usepackage{mathtools}
\usepackage{mathrsfs}

\usepackage[all]{xy}
\usepackage[toc,page]{appendix}
\usepackage{titlesec}
\usepackage{upgreek}

\usepackage{tocloft}

\usepackage{tocbasic}
\DeclareTOCStyleEntry[
beforeskip=.1em plus 1pt,
entryformat=\normalfont ,
pagenumberformat=\bfseries
]{tocline}{section}
\usepackage{etoolbox}
\patchcmd{\thebibliography}{%
	\section*{\refname}\@mkboth{\MakeUppercase\refname}{\MakeUppercase\refname}}{%
	\section*{\refname}}{}{}

\usepackage{ stmaryrd }

\newcounter{mainthm}

\newcounter{mainconj}

\newtheorem{thm}{Theorem}[section]

\theoremstyle{plain}
\newtheorem{lem}[thm]{Lemma}
\newtheorem{prop}[thm]{Proposition}
\newtheorem{cor}[thm]{Corollary}
\newtheorem{defn-thm}[thm]{Definition-Theorem}

\newtheorem{defn-lem}[thm]{Definition-Lemma}

\newtheorem{defn}[thm]{Definition}

\theoremstyle{definition}

\newtheoremstyle{rmk}
{5pt}
{5pt}
{}
{}
{\itshape}
{}
{.5em}
{}

\newtheorem{rmk}[thm]{Remark}

\newtheoremstyle{note}
{8pt}
{5pt}
{\itshape}
{10pt}
{\bfseries}
{}
{.5em}
{}

\theoremstyle{note}

\setlist[description]{font=
	\normalfont
	\itshape
	\space}

\usepackage{titletoc}

\titlecontents{section}
[0.72em]
{\scriptsize\bfseries\vspace{3pt}}%
{\thecontentslabel.\enspace}
{}
{\titlerule*[0.38pc]{.}\contentspage}%

\titlecontents{subsection}
[0em]
{\scriptsize \vspace{0pt}}%
{\qquad \quad \thecontentslabel.\enspace}
{}
{\titlerule*[0.5pc]{.}\contentspage}%

\titlecontents{subsubsection}
[0em]
{\footnotesize \vspace{1pt}}%
{\qquad \qquad \thecontentslabel.\enspace}
{}
{\titlerule*[0.5pc]{.}\contentspage}%

\titleformat{\subsubsection}[runin]{
	\bfseries
	\itshape\normalsize}{(\thesubsubsection) \ }{0em}{}[\mbox{ . } ]
\titlespacing{\subsubsection}
{0pt}
{2.25ex plus 1ex minus .2ex}
{0pt}

\newcommand{\id}{\mathrm{id}}

\newcommand{\one}{\mathds {1}}

\def\bar{\overline}
\def\hat{\widehat}

\def\^{\wedge}
\def\dscup{\dot\sqcup}

\def\C{\mathbb{C}}

\def\R{\mathbb{R}}
\def\Z{\mathbb{Z}}

\def\cH{\mathcal{H}}
\def\cM{\mathcal{M}}

\def\cE{\mathcal{E}}

\def\cO{\mathcal{O}}
\def\cH{\mathcal{H}}

\def\qq{\mathfrak{q}}
\def\ll{\mathfrak{l}}
\def\mm{\mathfrak{m}}
\def\one{\mathbf{1}}
\def\gg{\mathfrak{g}}

\def\d{\delta}

\def\e{\epsilon}

\def\k{\kappa}

\def\L{\Lambda}

\def\r{\rho}
\def\s{\sigma}

\def\O{\Omega}

\def\p{\partial}
\def\i{\iota}

\begin{document}
	\setlength{\parindent}{15pt}	\setlength{\parskip}{0em}

\title{Non-commutative calculus and Getzler-Gauss-Manin connections for Open-closed Homotopy Algebras}

\renewcommand{\thanks}[1]{\footnote{#1}}	

\author[Z. Yu]{Zekai Yu \thanks{\tiny Tsinghua University (Qiuzhen College), Beijing 100084, China. \quad Email: \texttt{yuzk23@mails.tsinghua.edu.cn}}}

\begin{abstract} {\textsc{abstract}: We establish the calculus structure on Hochschild invariants of open-closed homotopy algebras. We further define the Getzler-Gauss-Manin connection and show that it is flat up to chain homotopy on the open-closed periodic cyclic chain complex.}  

\end{abstract}

\maketitle

\section{Introduction}

Non-commutative calculus was introduced by Tamarkin and Tsygan \cite{tamarkintsy2005ring,GelfandDaletskiiTsygan1990,Tsygan2004} as a generalization of the classical Cartan calculus on a smooth manifold $M$. The latter is essentially an interplay among the algebra of smooth functions $A=C^\infty(M)$, the graded vector space of differential forms $\O^\bullet(M)$ and the graded commutative, graded Lie algebra of polyvector fields $\mathcal V^\bullet(M)$, with the following remarkable identities for vector fields $X, Y$ 
\begin{align}
    L_X=d\i_X+\i_Xd,\quad [L_X,\i_Y]=\i_{[X,Y]},\quad [L_X,L_Y]=L_{[X,Y]}\label{eq_cartan_id}
\end{align}
It was proven in \cite{HochschildKostantRosenberg1962} that the algebra of polyvector fields $\mathcal V^\bullet$ can be recovered from $A$ by the Hochschild cohomology of $A$, denoted $HH^\bullet(A,A)$, while the space of differential forms can be identified with the Hochschild homology of $A$, denoted $HH_\bullet(A,A)$. The natural goal of non-commutative calculus is hence to reconstruct the interactions between these Hochschild invariants from the datum of a (not necessarily commutative) associative algebra $A$, together with what is known as a Connes operator $B$ which is a non-commutative analogue of de Rham differential. The word "calculus" in particular implies the Cartan-type identities \eqref{eq_cartan_id}. The notion of a non-commutative calculus will be recalled precisely in Definition \ref{def_cal}.

%
In a similar vein, a unital $A_\infty$-algebra carries a calculus structure \cite{Tsygan2004,Getzler2002CARTANHF,chen2022note}. 
An $A_\infty$-algebra $A$ also admits the associated Hochschild cohomology $HH^\bullet(A,A)$ which carries a natural Gerstenhaber algebra structure. In particular, it is both a graded commutative algebra and a graded Lie algebra. The Hochschild homology $HH_\bullet(A,A)$ as a graded vector space is defined as well with a natural Connes operator $B$ satisfying $B^2=0$. There is a well-known calculus structure on the Hochschild invariants of $A$ \begin{prop}[\cite{chen2022note}, see also \cite{Getzler2002CARTANHF,Tsygan2004}]
    The Hochschild homology $HH_\bullet(A,A)$ carries 
    \begin{itemize}
        \item   a graded module structure over the graded commutative algebra $HH^\bullet(A,A)$;
        \item   a graded module structure over the graded Lie algebra $HH^\bullet(A,A)$
    \end{itemize}
    which are compatible, and further compatible with the Connes operator in the sense of Definition \ref{def_cal}. 
\end{prop}
\vspace{1em}
With the Connes operator $B$, one can define the (periodic) cyclic homology $HP_\bullet(A)$ as a non-commutative analogue of de Rham cohomology $H^\bullet_{dR}(M;\R)$ on a smooth manifold. Just as the classical Cartan calculus underlies the Gauss-Manin connection and Kodaira-Spencer theory for algebraic or holomorphic families of complex manifolds, non-commutative calculus provides the corresponding framework for the Getzler-Gauss-Manin connection associated with families of $A_\infty$-algebras. In \cite{Getzler2002CARTANHF}, Getzler established the following
\begin{prop}
    There is a natural connection on the periodic cyclic chain complexes associated to a family of $A_\infty$-algebras, which is flat up to chain homotopy.
\end{prop}
On the other hand, an open-closed homotopy algebra, or OCHA, is a structure describing how an $L_\infty$-algebra deforms an $A_\infty$-algebra in the curved sense. Inspired by classical open-closed string field theory, it was made precise as an algebraic framework by Kajiura and Stasheff \cite{Kajiura_2006ocha,Kajiura_2006homalg}. OCHAs are also related to the operator $\qq$ appearing in the geometric theory of bulk deformations in the story of mirror symmetry between compact toric manifolds and Landau-Ginzburg models, as developed by Fukaya-Oh-Ohta-Ono \cite{fukaya_ms_compacttoric,Fukaya_toric1,fukaya_toric2,FOOO}. The literature on this subject is extensive, and we do not attempt to give a complete list of references here. 

It is then reasonable to expect that a unital OCHA, as an algebraic structure incorporating simultaneous open/closed deformations, admits a calculus analogous to that of a unital $A_\infty$-algebra, and that a family of OCHAs admits a notion of Getzler--Gauss--Manin connection. The latter is supposed to encode deformations of closed sectors as well. In this paper, we verify the calculus structure by explicit computations. Moreover, we construct the natural Getzler--Gauss--Manin connection for a family of unital OCHAs and show that it is flat up to chain homotopy.

In fact, for an OCHA, we can define analogously the notion of open-closed Hochschild cohomology $HH^\bullet(\gg;A,A)$ which carries a natural Gerstenhaber algebra structure \cite{yuan2024openclosedstringanaloguehochschild}, and the open-closed Hochschild homology $HH_\bullet(\gg;A,A)$ which is itself a graded vector space \cite{wang2025openclosedhochschildhomologyrelative}. There is also a Connes operator $B$ on Hochschild homology. We prove the following result in this paper


\begin{thm}[= Corollary \ref{cor_hom}]
    The Hochschild homology $HH_\bullet(\gg;A,A)$ carries 
    \begin{itemize}
        \item   a graded module structure over the graded commutative algebra $HH^\bullet(\gg;A,A)$;
        \item   a graded module structure over the graded Lie algebra $HH^\bullet(\gg;A,A)$
    \end{itemize}
    which are compatible, and further compatible with the Connes operator.
\end{thm}
\vspace{.5em}
 In particular, the Hochschild invariants of an OCHA carry a calculus structure.
 
Analogous to the case of $A_\infty$-algebras, also as an application of the calculus structure we have developed, we define the Getzler-Gauss-Manin connection on the open-closed periodic cyclic chain complexes of a family of open-closed homotopy algebras,  by computations similar to those in \cite{Getzler2002CARTANHF}
\begin{thm}[= Theorem \ref{thm_GGM}]
   There is a natural Getzler-Gauss-Manin connection on the periodic cyclic chain complexes associated to a family of OCHAs, which is flat up to chain homotopy.
\end{thm}
\vspace{.5em}
We now briefly discuss potential applications and future directions. Non-commutative calculus and Getzler-Gauss-Manin connections play an important role in homological mirror symmetry and non-commutative Hodge theory. By incorporating the bulk $L_\infty$-algebras, OCHAs may provide a richer framework for the story. The structures studied here also point toward applications beyond the standard mirror-symmetric setting. For Calabi--Yau categories and cyclic open--closed homotopy algebras, they provide a natural language for BV structures, noncommutative Poincaré duality, and open--closed topological field theories. 

In the scenario of mirror symmetry between compact toric varieties and Landau-Ginzburg models, we expect a structure that is similar to OCHAs on the B-side by Kontsevich's formality result for the differential graded Lie algebra structure on the Hochschild cochain complex. In representation-theoretic directions, one may also hope for braided or quantum-group-equivariant versions of open-closed calculus, as Hochschild and cyclic invariants admit braided analogues for algebra objects in braided monoidal categories. These possibilities suggest that open--closed calculus structures, together with their Getzler-Gauss-Manin connections, form a useful bridge between homotopy algebra, Calabi--Yau geometry, mirror symmetry, and quantum algebra. 

This paper is organized as follows. In Section \ref{sec_pre} we present preliminaries and background. In particular, we fix our sign conventions and recall the relevant structures associated with OCHAs. We also set up a sign rule to simplify the practical computations involving signs. In Section \ref{sec_cal} we construct the calculus structure on OCHAs by explicitly writing down the operators and verifying the required identities. In Section \ref{sec_ggm} we establish the results on Getzler--Gauss--Manin connections for a family of OCHAs.
 
\subsection*{Acknowledgement} The author would like to thank Hang Yuan for proposing this problem, for carefully reading the manuscript and for giving valuable suggestions. He would also like to thank Si Li for useful discussions on the subject of cyclic theory. The author acknowledges the assistance of OpenAI's GPT-5.6 Thinking in pointing out several subtleties and suggesting the filtration-based proof in place of an earlier normalization argument. 

\section{Preliminaries} \label{sec_pre}

 Background can be found in \cite{yuan2024openclosedstringanaloguehochschild,wang2025openclosedhochschildhomologyrelative}, see also \cite{yuan2025cyclicbracerelationbv}.  We fix a field $k$ of characteristic zero. We use notation $[a]$ to denote the ordered set $\{1,...,a\}$. We use notation $I_1\sqcup I_2\sqcup ... I_n=[a]$ to denote the partition of $[a]$ into $n$ disjoint subsets where each block $I_j$ is endowed with strictly increasing order induced from $[a]$. We write $J_1\dscup J_2\dscup ...J_m=[d]$ to denote an order-preserving partition of $[d]$; this means that each $J_j$ is a consecutive interval in $[d]$, and the ordered concatenation of the blocks recovers the original order on $[d]$. A priori, each block could be empty. In this notation, a tensor $x_{[d]}\in V^{\otimes d}$ could be written as $x_{J_1}\otimes ...\otimes x_{J_k}$, where $J_1\dscup J_2...\dscup J_k=[d]$.

We follow the grading convention as in \cite{keller2001introductionainfinityalgebrasmodules,yuan2024openclosedstringanaloguehochschild}. We use cohomological degree for Hochschild chain complexes. We make our sign convention clear. The degree of an element $x$ of a graded vector space $V$ is denoted $\deg(x)$. A homogeneous multi-linear map $D:V^{\otimes n}\to V$ admits its own notion of degree $\deg (D)$ such that the total degree is additive
\[\deg(D)+\sum_{i=1}^d\deg(y_i)=\deg(D(y_1\otimes  ...\otimes y_d))\]
This renders a grading on $\hom(V^{\otimes n},V)$. The graded symmetric case is similar, just replacing $V^{\otimes n}$ by $V^{\wedge n}$. In the current setup, we frequently use the notion $|\cdot|$ of shifted degree, $|x|=\deg(x)-1$. In other words, it is the degree of the suspension $sx\in V[1]$. An element $a\in V^{\otimes n}$ has shifted degree $|a|=\deg(a)-n$. A homogeneous multi-linear map $D$ is also endowed with the shifted degree $|D|$ such that the total shifted degree is additive
\[|D|+\sum |y_i|=|D(y_1\otimes ...\otimes y_d)|\]
Hence we conclude that $|D|=\deg(D)+d-1$. In this paper, we will only make use of the shifted degree of an element of a graded vector space $V$ or an operator thereof, which will be the graded vector space underlying an $A_\infty$-algebra $A$, an $L_\infty$-algebra $\gg$, or an open-closed homotopy algebra (OCHA) $(\gg,A)$.

\subsection{Background on open-closed homotopy algebras}
The notion of OCHA was introduced by Kajiura and Stasheff in \cite{Kajiura_2006homalg,Kajiura_2006ocha}. An OCHA is a tuple of structures $(\gg,A;\ll,\qq)$, where $(\gg,\ll)$ is an $L_\infty$-algebra, in the sense that $\gg$ is a graded vector space, and there is a sequence of multi-linear maps 
\[\ll_n:\gg^{\wedge n}\to \gg,\quad n\in \Z_+\]
of shifted degree $1$, satisfying the $L_\infty$-relation: For any $y_{[a]}\in \bigoplus_n \gg^{\wedge n}$,  
\[\sum_{I_1\sqcup I_2=[a]}(-1)^\e \ll_{|I_1|+|I_2|} (\ll_{|I_1|}(y_{I_1})\wedge y_{I_2})=0\]
where $\e$ is determined by the equality $y_{[a]}=(-1)^\e y_{I_1}\wedge y_{I_2}$. In a similar way, $(A,\qq_{0,-})$ is an $A_\infty$-algebra in the sense that $A$ is a graded vector space and there is a sequence of multi-linear maps 
\[\qq_{0,m}:A^{\otimes m}\to A,\quad m\in \Z_+\]
of shifted degree 1, satisfying the $A_\infty$-relation: For any $x_{[d]}\in \bigoplus_m A^m$, 
\[\sum _{J_1\dscup J_2\dscup J_3=[d]}(-1)^\ast \qq_{0,|J_1|+|J_3|+1}(x_{J_1}\otimes \qq_{0,|J_2|}(x_{J_2})\otimes x_{J_3})=0\]
where $\ast=|x_{J_1}|$ is given by the Koszul sign rule for the shifted degree. 

Given these structures, an OCHA consists of the additional data of $\qq_{n,m}:\gg^{\^ n}\otimes A^{\otimes m}\to A$ for $n>0,m\geq 0$, all of shifted degree $1$, satisfying the OCHA relation
\begin{align}
    \sum _{\substack{I_1\sqcup I_2=[a] \\J_1\dscup J_2\dscup J_3=[d]}}(-1)^{\e}\qq(\ll(y_{I_1})\wedge y_{I_2};x_{[d]})=\sum_{\substack{I_1\sqcup I_2=[a] \\J_1\dscup J_2\dscup J_3=[d]}} (-1)^{\clubsuit}\qq(y_{I_1};x_{J_1}\otimes\qq(y_{I_2};x_{J_2})\otimes x_{J_3})
\end{align}
where the sign is determined by Koszul rule: $y_{[a]}=(-1)^\e y_{I_1}\wedge y_{I_2}$, $\clubsuit=\e+|y_{I_2}||x_{J_1}|+|x_{J_1}|+|y_{I_1}|$. We emphasize that the $\qq_{n,0}$ terms make this structure nontrivial: the closed sector could potentially introduce a curvature to the $A_\infty$-algebra as indicated in \cite{fukaya_ms_compacttoric,fukaya_toric2}. We have suppressed the subscript of structural morphisms when the input is clear; in other words, we take $\qq=\sum_{n,m}\qq_{n,m}$. A semi-colon is placed to separate the closed ($L_\infty$) inputs and the open $(A_\infty)$ inputs of the structural morphism $\qq$. This relation is sometimes abbreviated as $\hat \ll (\qq)=\qq\{\qq\}$, using the brace notation, see\cite{yuan2024openclosedstringanaloguehochschild}.

A unit plays a special role. We recall that a unit in an $A_\infty$-algebra $A$ is an element $\one$ such that $\mm_k(a_1,\one,a_2)=0$ for $k> 2$, $\mm_1(\one)=0$, $\mm_2(\one,a)=a$, $\mm_2(a,\one)=(-1)^{|a|+1}a$. Slightly abusing notation, OCHA has a unit $\one$ in the sense that $\qq_{n,k}(y;x_1,\one,x_2)=0$ for $k>2,n\geq 0$, $\qq_{n,1}(y;\one)=0$ for $n\geq 0$, and $\qq_{n,2}(y;\one,x)=\qq_{n,2}(y;x,\one)=0$ for $n>0$. The only non-vanishing item is $\qq_{0,2}(\one,x)=(-1)^{|x|+1}\qq_{0,2}(x,\one)=x$.

\subsection{Open-closed Hochschild homology and cohomology}

For an OCHA $(\gg,A)$, we define the normalized open-closed Hochschild chain complex $C_\bullet(\gg;A,A)\coloneqq \bigoplus_{l,k\geq 0}\gg^{\wedge l}\otimes A\otimes \bar A^{\otimes k}$, where $\bar A=A/k\cdot \one $. The grading is given by the shifted degree $|-|$, or in other words the natural grading on $\bigoplus_{l,k\geq 0}\gg[1]^{\wedge l}\otimes A[1]\otimes \bar A[1]^{\otimes k}$. We also define the un-normalized Hochschild chain complex, $C^{un}_\bullet(\gg;A,A)\coloneqq \bigoplus_{k,l\geq 0} \gg[1]^{\wedge l}\otimes A[1]^{\otimes (k+1)}$ with the natural grading. There is an operator of degree 1 as follows

\begin{defn}
    $b:C^{un}_\bullet(\gg;A,A)\to C^{un}_{\bullet+1}(\gg;A,A)$ 
     \begin{align}
         b(y_{[a]}\otimes x_{[d]})\coloneqq &\sum_{\substack{L\dscup M\dscup R=[d],L\neq \phi\\I_1\sqcup I_2=[a]}}(-1)^{\ast_1} y_{I_1}\otimes x_L\otimes \qq(y_{I_2};x_M)\otimes x_R+\sum_{\substack{L\dscup M\dscup R=[d],L\neq \phi\\I_1\sqcup I_2=[a]}}(-1)^{\ast_2}y_{I_1}\otimes \qq(y_{I_2};x_R\otimes x_L)\otimes x_M\\ -&\sum_{I_1\sqcup I_2=[a]}(-1)^\e \ll(y_{I_1})\wedge y_{I_2}\otimes x_{[d]}
     \end{align}
    Here $\ast_1=\e+|x_L|+|y_{I_1}|+|y_{I_2}||x_L|$, $\ast_2=\e+|y_{I_1}|+|x_R|(|x_M|+|x_L|)$, $\e$ is determined by $y_{[a]}=(-1)^\e y_{I_1}\wedge y_{I_2}$. 
\end{defn}
Note that the index set $M$ could be empty. It can be proved that $b^2=0$, see also \cite{wang2025openclosedhochschildhomologyrelative}, hence $(C^{un}_\bullet,b)$ is a complex. The cohomology of this complex is called open-closed Hochschild homology, $HH_\bullet(\gg;A,A)\coloneqq H^\bullet(C^{un}_\bullet(\gg;A,A),b)$. However we will mostly work with the normalized Hochschild chain complex.
\begin{prop}
    The Hochschild differential $b$ defined above descends to a differential $\bar b$ on $C_\bullet(\gg;A,A)$. In fact, there is a chain map $\pi:C^{un}_\bullet(\gg;A,A)\to C_\bullet(\gg;A,A)$ that fits in the following commutative diagram.
\[\xymatrix{
C^{un}_\bullet(\gg;A,A) \ar[r]^{b} \ar[d]_{\pi}  & C^{un}_\bullet(\gg;A,A) \ar[d]^{\pi} \\
C_\bullet(\gg;A,A) \ar[r]_{\bar b}  & C_\bullet(\gg;A,A)}\]
\end{prop}
\begin{proof}
    We use a classical construction from \cite[Chapter 8, 9]{Weibel1960AnIT}, initially for associative algebras. This is later adapted for non-unital $A_\infty$-categories in e.g. \cite[Section 1.4]{loday1992cyclic} and \cite[Section 3.5]{Sheridan_2019}. We define the graded vector space $D_\bullet(\gg;A,A)$ to be the subspace of $C^{un}_\bullet(\gg;A,A)$ generated by tensors with a unit not in the first slot in the open sector input. A generator can be written uniquely as $w=y_{[a]}\otimes x_L \otimes \one \otimes x_R$, such that $L$ is non-empty and $x_R$ contains no unit as a factor. The Hochschild differential $b$ restricts to a differential on $D_\bullet(\gg;A,A)$, which makes it into a subcomplex of $C^{un}_\bullet(\gg;A,A)$. Indeed, one computes that 
    \begin{align}
        b(w)=&\sum_{\substack{L_1\dscup L_2\dscup L_3=L, I_1\sqcup I_2=[a]\\L_1\neq \phi}} (-1)^{\e+|y_{I_1}|+|x_{L_1}|+|y_{I_2}||x_{L_1}|}y_{I_1}\otimes x_{L_1}\otimes \qq(y_{I_2};x_{L_2})\otimes x_{L_3}\otimes \one \otimes x_R \label{b2.4} \\
        +&\sum_{R_1\dscup R_2\dscup R_3=R,I_1\sqcup I_2=[a]} (-1)^{\e+|y_{I_1}|+|x_L|+|x_{R_1}|+|y_{I_2}|(|x_L|+1+|x_{R_1}|)+1}y_{I_1}\otimes x_L\otimes \one \otimes x_{R_1}\otimes \qq(y_{I_2};x_{R_2})\otimes x_{R_3}\label{b2.5}\\
        +&\sum_{\substack{L_1\dscup L_2=L,L_1\neq \phi\\ R_1\dscup R_2=R,I_1\sqcup I_2=[a]}} (-1)^{\e+|x_{R_2}|(|x_L|+|x_{R_1}|+1)+y_{I_1}}y_{I_1}\otimes \qq (y_{I_2};x_{R_2}\otimes x_{L_1})\otimes x_{L_2}\otimes \one \otimes x_{R_1}\label{b2.6}\\
        +&\sum_{I_1\sqcup I_2=[a]} (-1)^{\e+1}\ll(y_{I_1})\wedge y_{I_2}\otimes x_L\otimes \one \otimes x_R\label{b2.7}\\
        +&(-1)^{|y_{[a]}|+|x_{L_1}|}y_{[a]}\otimes x_{L_1}\otimes \qq_{0,2}(x_{l}\otimes \one)\otimes x_R \quad \text{  (if $|L|\geq 2$, $L\eqqcolon L_1\dscup \{l\}$)}\label{b2.8}\\
        +& (-1)^{|y_{[a]}|+|x_L|}y_{[a]}\otimes x_L\otimes \qq_{0,2}(\one\otimes x_r)\otimes x_{R_1} \quad \text{ (if $|R|\geq 1$, $R\eqqcolon\{r\}\dscup R_1$)}\label{b2.9}\\
        +& (-1)^{|y_{[a]}|} y_{[a]}\otimes \qq_{0,2}(x_L\otimes \one)\otimes x_R \quad \text{ (if $|L|=1$)}\label{b2.10}\\
        +& (-1)^{y_{[a]}} y_{[a]}\otimes \qq_{0,2}(\one \otimes x_{L_1})\otimes x_{L_2} \quad \text{ (if $R=\phi$)}\label{b2.11}
    \end{align}
   where $y_{[a]}=(-1)^\e y_{I_1}\wedge y_{I_2}$ as usual. Now observe that \eqref{b2.8} and \eqref{b2.10} combine into a copy of $(-1)^{\e+|y_{[a]}|+|x_L|+1}w$, \eqref{b2.9} and \eqref{b2.11} combine into a copy of $(-1)^{\e+|y_{[a]}|+|x_L|}w$, which cancel precisely. We have applied the unitality $\qq_{0,2}(\one,-)=(-1)^{|-|+1}\qq_{0,2}(-,\one)=\id$. The remaining terms \eqref{b2.4},\eqref{b2.5},\eqref{b2.6},\eqref{b2.7} all sit in $D_\bullet(\gg;A,A)$.
   
   Now the quotient complex $C^{un}_\bullet(\gg;A,A)/D_\bullet(\gg;A,A)$ is just the conventional normalized Hochschild chain complex, $C_\bullet(\gg;A,A)$, and we take $\pi$ to be the natural quotient map.    
\end{proof}
Slightly abusing notation, we will write $b$ for $\bar b$. For associative algebras, it is known that the $\pi$ is a quasi-isomorphism, see e.g. \cite{Weibel1960AnIT}. The same argument works for $A_\infty$-algebras and OCHAs as well. Hence one may freely use either $C_\bullet^{un}(\gg;A,A)$ or $C_\bullet(\gg;A,A)$ as the chain model of Hochschild homology. \begin{lem}
    The complex $(D_\bullet(\gg;A,A),b|_{D_\bullet})$ is acyclic.\end{lem}
\begin{proof}
        We define the following operator (keeping in mind the unique presentation of $w\in D_\bullet$) \[\Gamma:D_\bullet(\gg;A,A)\to D_{\bullet-1}(\gg;A,A):\quad y_{[a]}\otimes x_L\otimes \one \otimes x_R\mapsto (-1)^{|y_{[a]}|+|x_L|+1}y_{[a]}\otimes x_L\otimes \one \otimes \one \otimes x_R \]
     Unlike in the case of classical associative algebras or $A_\infty$-algebras, it is not true that $b\circ \Gamma+\Gamma\circ b=\id_{D_\bullet(A,A)}$. However, we can bypass this using a spectral sequence argument associated to the natural closed-arity filtration on $C^{un}_\bullet$ and hence $ D_\bullet$
     \[F^pD\coloneqq \bigoplus_{k\geq 0}\bigoplus_{l=0}^p \gg^{\wedge l}\otimes A^{\otimes (k+1)}\cap D\]
     It is clear that $b(F^pD)\subset F^pD$. We write $b=b_0+\tilde b$, where $b_0$ consists of operations $\qq_{0,k}$ and $\ll_1$, $\tilde b$ packages the others. The claim is that $b_0\circ \Gamma+\Gamma\circ b_0=\id$ on $\mathrm{gr}^F_p(D)$. To this end, compute    
    \begin{align}
        b\circ \Gamma(w) &= \sum_{\substack{L_1\dscup L_2\dscup L_3=L,L_1\neq \phi\\I_1\sqcup I_2=[a]}} (-1)^{\e+|x_L|+1+|x_{L_1}|+|y_{I_2}|(|x_{L_1}|+1)}y_{I_1}\otimes x_{L_1}\otimes \qq(y_{I_2};x_{L_2})\otimes x_{L_3}\otimes \one \otimes \one \otimes x_R\\
+&\sum_{\substack{L_1\dscup L_2\dscup L_3=L,L_1\neq \phi\\I_1\sqcup I_2=[a]}} (-1)^{\e+|y_{I_2}||x_{L}|}y_{I_1}\otimes x_{L}\otimes \one \otimes\qq(y_{I_2})\otimes \one \otimes x_R \label{eq_acyclic_openarity0}\\        
        +&\sum_{L_1\dscup L_2=L,L_1\neq \phi, |L_2|=1} (-1)^{|x_L|+1+|x_{L_1}|}\textcolor{ForestGreen}{y_{[a]}\otimes x_{L_1}\otimes \qq(x_{L_2}\otimes \one)\otimes \one \otimes x_R} \label{green2.17}\\
        +&(-1)^{1} \textcolor{RubineRed}{y_{[a]}\otimes x_L\otimes \one\otimes x_R}\label{red2.18}\\
        +&\sum_{R_1\dscup R_2=R,|R_1|=1} (-1)^{0} \textcolor{NavyBlue}{y_{[a]}\otimes x_L\otimes \one \otimes \qq(\one \otimes x_{R_1})\otimes x_{R_2}}\label{blue2.19}\\
        +&\sum_{\substack{R_1\dscup R_2\dscup R_3=R\\I_1\sqcup I_2=[a]}} (-1)^{\e+1+|x_{R_1}|+|y_{I_2}|(|x_L|+|x_{R_1}|+1)}y_{I_1}\otimes x_L\otimes \one \otimes \one \otimes x_{R_1}\otimes \qq(y_{I_2};x_{R_2})\otimes x_{R_3}\\
        +&\sum_{\substack{L_1\dscup L_2=L,R_1\dscup R_2=R\\I_1\sqcup I_2=[a],L_1\neq \phi}} (-1)^{\e+|y_{I_2}|+|x_L|+1+|x_{R_2}|(|x_L|+|x_{R_1}|)} y_{I_1}\otimes \qq(y_{I_2};x_{R_2}\otimes x_{L_1})\otimes x_{L_2}\otimes \one \otimes \one \otimes x_{R_1}\\
        +& \sum_{|L|=1}(-1)^{|x_L|+1}\textcolor{ForestGreen}{y_{[a]}\otimes \qq(x_L\otimes \one) \otimes \one \otimes x_{R}}\label{green2.22}\\
        +& \sum_{R=\phi,L_1\dscup L_2=L,|L_1|=1}(-1)^{0}\textcolor{NavyBlue}{y_{[a]}\otimes \qq(1\otimes x_{L_1})\otimes x_{L_2}\otimes \one} \label{blue2.23}\\
        +&\sum_{I_1\sqcup I_2=[a]} (-1)^{\e+|x_L|+|y_{[a]}|}\ll(y_{I_1})\wedge y_{I_2}\otimes x_L\otimes \one \otimes x_R
    \end{align}
where $\e$ is such that $y_{[a]}=(-1)^\e y_{I_1}\wedge y_{I_2}$. We see that \eqref{green2.17} and \eqref{green2.22} combine into a copy of $y_{[a]}\otimes x_L\otimes \one \otimes x_R=w$, so do \eqref{blue2.19} and \eqref{blue2.23}. One copy cancels \eqref{red2.18}. \eqref{eq_acyclic_openarity0} drops the closed-arity.

\begin{align}
    \Gamma\circ b(w)=& \sum_{L_1\dscup L_2\dscup L_3=L,L_1\neq \phi} (-1)^{\e+|x_L|+|x_{L_1}|+|y_{I_2}|+|y_{I_2}||x_{L_1}|}y_{I_1}\otimes x_{L_1}\otimes \qq (y_{I_2};x_{L_2})\otimes x_{L_3}\otimes \one \otimes \one \otimes x_R\\
    +&\sum_{R_1\dscup R_2\dscup R_3=R} (-1)^{\e+|x_{R_1}|+|y_{I_2}|(1+|x_L|+|x_{R_1}|)}y_{I_1}\otimes x_L\otimes \one \otimes \one \otimes x_{R_1}\otimes \qq(y_{I_2};x_{R_2})\otimes x_{R_3}\\
    +&\sum_{L_1\dscup L_2=L,R_1\dscup R_2=R,L_1\neq \phi} (-1)^{\e+|x_{R_2}|(|x_L|+|x_{L_1}|)+|x_L|+|y_{I_2}|} y_{I_1}\otimes \qq(y_{I_2};x_{R_2}\otimes x_{L_1})\otimes x_{L_2}\otimes \one \otimes \one \otimes x_{R_1} \\
    +&\sum_{I_1\sqcup I_2=[a]}(-1)^{\e+|x_L|+|y_{[a]}|+1}\ll(y_{I_1})\wedge y_{I_2}\otimes x_L\otimes \one  \otimes x_R\\
    +&\text{ Terms that involve }\qq(y_{I_2})\label{eq_acyclic_openari0}
\end{align}
\eqref{eq_acyclic_openari0} drops the closed-arity. Apart from that, all terms with at least two units cancel. Terms with bulk insertions cancel. The net result is only $w$. This shows the claim. 

Now to prove the lemma, pick a cycle $z\in D_\bullet$. Suppose it is in $F^pD$. Acyclicity of the associated graded means that there exists an element $h_1\in F^pD$ such that $z-bh_1\in F^{p-1}D$ is a cycle. Repeating this process, we see that $z$ is a boundary, hence $D_\bullet$ is acyclic. This is equivalently a spectral sequence argument for the filtered complex $(F^pD^\bullet,b)$.
\end{proof}
    With the proposition, we may safely choose a splitting $A\cong \bar A\oplus k\cdot \one$. Then $b$ restricted on the normalized complex $C_\bullet(\gg;A,A)$ is computed by first applying the formula above, regarding $C_\bullet$ as a graded sub-vector space of $C_\bullet^{un}$ through the splitting, and then projecting to the normalized Hochschild chain complex.

Similarly, the normalized open-closed Hochschild cochain complex is defined as $$C^\bullet(\gg;A,A)\coloneqq \prod_{\substack{i,k\geq 0\\(i,k)\neq (0,0)}} \hom_k(\gg^{\wedge i}\otimes  {\bar A}^{\otimes j},A)$$ endowed with shifted degree, or in other words $ \prod_{i,j} \hom_k(\gg[1]^{\wedge i}\otimes  \bar A[1]^{\otimes j},A[1])$ endowed with the natural grading. It can be naturally identified with the graded sub-vector space of $$C_{un}^\bullet(\gg;A,A)\coloneqq \prod_{\substack{i,j\geq 0}\\(i,j)\neq (0,0)} \hom_k(\gg^{\wedge i}\otimes A^{\otimes j},A)$$ consisting of cochains that vanish if any open-sector input argument is the unit. The space of unnormalized open-closed Hochschild cochains is a complex, equipped with the following differential \cite{yuan2024openclosedstringanaloguehochschild}

\[\d D\coloneqq\qq\{D\}-(-1)^{|D|}D\{\qq\}+(-1)^{|D|}\hat \ll(D)\]

which has shifted degree $1$. One can check that it restricts to a differential on $C^\bullet(\gg;A,A)$ and so endows it with a chain complex structure. We define the open-closed Hochschild cohomology as the cohomology of the normalized complex $HH^\bullet(\gg;A,A)\coloneqq H^\bullet(C^\bullet(\gg;A,A),\d)$. It was shown in \cite{yuan2024openclosedstringanaloguehochschild,yuan2025cyclicbracerelationbv} to admit a natural BV structure, whose induced Gerstenhaber bracket is given by
\[[D,E]\coloneqq D\{E\}-(-1)^{|D||E|}E\{D\}\]
and the Yoneda product 
\[D\cup E\coloneqq (-1)^{|D|+1}\qq\{D,E\}\]
We now define the graded Lie bracket using composition of operators on $C_\bullet(\gg;A,A)$
\[[f,g]_\circ \coloneqq f\circ g-(-1)^{|f||g|}g\circ f\]

\subsection{The Connes operator}
In this section we define the Connes operator $B$ for an OCHA. Recall that we are working in the normalized open-closed Hochschild chain complex. We start from the Connes operator of a unital $A_\infty$-algebra. We take the normalized Hochschild chain complex $C_\bullet(A,A)$.
\begin{defn}
    The (normalized) Connes operator $B:C_\bullet(A,A)\to C_{\bullet-1}(A,A)$ is defined as
    $$B(a_{[d]})\coloneqq \sum_{\substack{L\dscup R=[d]\\L\neq \phi}}(-1)^{|a_L||a_R|}\one\otimes a_R\otimes a_L$$
    
    The shifted degree of $\one$ is $|\one|=-1$, hence $|B|=-1$.
\end{defn}
\begin{rmk} \label{rmk_A_inf_bB}
    As we work in the normalized Hochschild chain complex, it is obvious that $B^2=0$. Moreover, one can check directly that $B\circ b+b\circ B=0$, based on the property $\mm_2(\one,a)=a$, $\mm_2(a,\one)=(-1)^{|a|+1}a$. In more detail, expand 
    \begin{align}
        B\circ b(a_{[d]})=&\sum_{\substack{L\dscup M\dscup R_1\dscup R_2=[d]\\L\neq \phi}} (-1)^{|a_L|+|a_{R_2}|(|a_{[d]}|-|a_{R_2}|+1)}\one \otimes a_{R_2}\otimes a_L\otimes \mm(a_M)\otimes a_{R_1} \\
        +&\sum_{\substack{L_1\dscup L_2\dscup M\dscup R_2=[d]\\L_1\neq \phi}} (-1)^{|a_L|+(|a_{[d]}|-|a_{L_1}|)|a_{L_1}|}\one \otimes a_{L_2}\otimes \mm(a_M)\otimes a_R\otimes a_{L_1}\\
        +&\sum_{\substack{L\dscup M_1\dscup L_2\dscup R=[d]\\L\neq \phi}} (-1)^{(|a_{M_2}|+|a_R|)(|a_{[d]}|-|a_{M_2}|-|a_R|)+|a_{M_2}|}\one \otimes a_{M_2}\otimes \mm(a_R\otimes a_L)\otimes a_{M_1}
    \end{align}
    \begin{align}
        b\circ B(a_{[d]})=&\sum_{L\dscup R=[d],L\neq \phi} (-1)^{|a_L||a_R|}a_R\otimes a_L \label{eq2.7}\\
        +&\sum_{L\dscup R_1\dscup R_2\dscup R_3=[d],L\neq \phi} (-1)^{|a_L||a_R|+|a_{R_1}|+1}\one \otimes a_{R_1}\otimes \mm(a_{R_2})\otimes a_{R_3}\otimes a_L\\
        +&\sum_{L_1\dscup L_2\dscup R_1\dscup R_2=[d],L_1\neq \phi} (-1)^{|a_L||a_R|+|a_{R_1}|+1}\one \otimes a_{R_1}\otimes \mm(a_{R_2},a_{L_1})\otimes a_{L_2}\\
        +&\sum_{L_1\dscup L_2\dscup K_3\dscup R=[d],L_1\neq \phi} (-1)^{|a_L||a_R|+|a_R|+|a_{L_1}|+1}\one \otimes a_R\otimes a_{L_1}\otimes \mm(a_{L_2})\otimes a_{L_3}\\
        +&\sum_{L\dscup \{l\}\dscup R=[d]} (-1)^{|a_L|(|a_R|+|a_l|)+1}a_l\otimes a_R\otimes a_L\label{eq2.11}
    \end{align}
    the right hand side of \eqref{eq2.7} cancels \eqref{eq2.11}. Other terms cancel in pairs. The same computation applies to unital curved $A_\infty$-algebras.
\end{rmk}     
For an OCHA $(\gg,A)$, we define the Connes operator similarly
\begin{defn}
    The Connes operator $B:C_\bullet(\gg;A,A)\to C_{\bullet-1}(\gg;A,A)$ is defined as follows
    $$B(y_{[a]};x_{[d]})\coloneqq \sum_{\substack{L\dscup R=[d]\\L\neq \phi}}(-1)^{|y_{[a]}|+|a_L||a_R|}y_{[a]}\otimes \one\otimes a_R\otimes a_L$$
    More precisely, it is the induced operator in the normalized Hochschild chain complex.
\end{defn}
This operator obviously squares to zero. In a way similar to the (curved) $A_\infty$-case, one checks that $B\circ b+b\circ B=0$.

\subsection{A sign rule}
In this subsection we form a sign rule in order for our proof to be more readable.

Given a pure tensor $w$ in $C_\bullet(\gg;A,A)$, we denote by $\mathsf w$ its associated word, i.e. forgetting all operating brackets. For instance, the associated word of $y_{I_1}\wedge y_{I_2\sqcup I_3}\otimes x_R\otimes x_L\otimes D(\ll(y_{I_4});x_M)$ is $y_{I_1}\wedge y_{I_2\sqcup I_3}\otimes x_R\otimes x_L\otimes D\otimes \ll\otimes y_{I_4}\otimes x_M$. The Koszul sign rule w.r.t. the shifted degree is the rule of additional signs arising from exchanging adjacent letters. If two words differ only by a sequence of such exchange of adjacent letters, a natural question is to ask what signs appear for a sequence of adjacent exchange. The Koszul sign rule implies the following 

\begin{lem}
    The sign appearing in permuting a word $\mathsf w_1$ to another word $\mathsf w_2$ is independent of the sequence of adjacent exchange. It depends only on the words themselves.
\end{lem}
\begin{proof}
    Let $\mathsf w_1$ be a word of homogeneous letters $\mathsf w_1=p_1p_2...p_n$, of (shifted) degree $|p_i|$, $\mathsf w_2=p_{\s(1)}p_{\s(2)}...p_{\s(n)}$ for some permutation $\sigma$. The sign is controlled by a sequence of moves as follows, i.e. exchange of adjacent letters,
    \[...p_ip_j...\leadsto  (-1)^{|p_i||p_j|}...p_jp_i...\]
  Each move involves only two letters and contributes a summand in the exponent of $-1$. A sequence of moves contributes a list $N_{ij}$ recording the number of times $p_i$ and $p_j$ are exchanged for each ordered pair $i<j$. The total Koszul sign is $(-1)^{\sum_{i<j} N_{ij}|p_i||p_j|}$, so what matters is $N_{ij} \text { mod } 2$. This is determined only by the relative position of $p_i$ and $p_j$ in $\mathsf w_2$: If $\s^{-1}(i)<\s^{-1}(j)$, then $N_{ij}$ is even, otherwise it is odd. Hence the total Koszul sign is determined only by the word $\mathsf w_2$ but not the concrete sequence of moves.
\end{proof}

This lemma enables the following definition

\begin{defn}
    The \textit{canonical Koszul sign} $\k(\mathsf{w})$ of a word $\mathsf{w}$ is the sign appearing in any sequence of adjacent exchange, such that the final word is $\mathsf w_2=\mathsf w$, and the initial word $\mathsf w_1$ is given by $\mathsf z\otimes y_{[a]}\otimes x_{[d]}$. Here $\mathsf z$ is the tensor product of operations like $\qq,D,E,\one\otimes $ and $\ll$, put with their relative order in $\mathsf w$. 
\end{defn}
\vspace{.5em}
Slightly abusing notation, we write $\k(w)$ for the canonical Koszul sign of the associated word of $w$. We may omit the argument when the latter is clearly the word that follows the sign.

\paragraph{Example}
Let $w=y_{I_1}\wedge y_{I_2\sqcup I_3}\otimes x_R\otimes \one \otimes x_L\otimes D(\ll(y_{I_4});x_M)$, $\mathsf w=y_{I_1}\wedge y_{I_2\sqcup I_3}\otimes x_R\otimes \one \otimes x_L\otimes D\otimes \ll \otimes y_{I_4}\otimes x_M$. Then $\mathsf z=\one \otimes D\otimes \ll$, and $\mathsf w_1=\one \otimes D\otimes \ll \otimes y_{[a]}\otimes x_{[d]}$. We have 

\begin{align}
\k(y_{I_1}\wedge y_{I_2\sqcup I_3}\otimes x_R\otimes \one\otimes x_L\otimes D(\ll(y_{I_4});x_M))&=\notag \\
 \e+|x_R|(|x_L|+|x_M|)+|y_{I_4}|(|x_R|+|x_L|)&+(|y_{I_1}|+|y_{I_{2\sqcup 3}}|+|x_L|+|x_R|)|D|+ |x_L|\text{ mod } 2
\end{align}

where $y_{[a]}=(-1)^\e y_{I_1}\wedge y_{I_2\sqcup I_3}\wedge y_{I_4}$. This can be obtained, say, via first re-order subwords $y_{[a]}\otimes x_{[d]}\leadsto (-1)^{\e+|x_R|(|x_L|+|x_M|)+|y_{I_4}|(|x_R|+|x_L|)}y_{I_1}\wedge y_{I_2\sqcup I_3}\otimes x_R\otimes x_L\otimes y_{I_4}\otimes x_M$, then moving the letters in $\mathsf z$ rightward to the correct places.

Now comes the crucial observation. The canonical Koszul sign is defined with respect to a fixed initial word, which need not coincide with the word occurring in a given computation. The discrepancy, however, has only one source: adjacent exchange in $\mathsf z$. These extra signs are exactly those produced by moving one operation past another. This is the nontrivial contribution that we record separately from the universal Koszul factor $\k$.

In this way, the sign problem can be solved much more efficiently. The proof of identities is largely reduced to the problem of correctly classifying the words.

\section{The calculus structure on OCHA} \label{sec_cal}

In this section we establish the calculus structure on OCHAs. To ease the readers from heavy computations, let us say that the proof is not a collection of \textit{ad hoc} calculations. Rather it follows from e.g. the brace formalism similar to $A_\infty$-algebras \cite{Getzler2002CARTANHF}. The novelty is that the usage of OCHA structural relation instead of $A_\infty$ structural relation produces terms involving the closed sector structural morphisms $\ll_k$, which miraculously cancel. 

First we recall the notion of a calculus \cite{tamarkintsy2005ring} 
\begin{defn} \label{def_cal}
    A calculus is a pair of Gerstenhaber algebra $(\mathcal V^\bullet,\cup,[-,-])$ and a graded vector space $\O^\bullet$ together with 
    \begin{itemize}
        \item A structure of a graded module over the graded commutative algebra $(\mathcal V^{\bullet-1},\cup)$ on $\O^{\bullet}$, with the corresponding action denoted by $\i_D$, $D\in \mathcal V^\bullet$.
        \item A structure of a graded module over the graded Lie algebra $(\mathcal V^{\bullet},[-,-])$ on $\O^{\bullet}$, with the corresponding action denoted by $L_D$, $D\in \mathcal V^\bullet$ such that 
\[[L_D,\i_E]_\circ=(-1)^{|D|}\i_{[D,E]}\]
and $L_{D\cup E}=L_D\circ \i_E+(-1)^{|D|+1}\i_D\circ L_E$. Here $|\cdot|$ denotes the graded degree in both $\mathcal V^\bullet$ and $\O^\bullet$.
\item An operator $B$ of degree $-1$ on $\O^{\bullet}$ such that $B^2=0$ and $[B,\i_D]_\circ=L_D$.
    \end{itemize}
\end{defn}
The key is to define the correct action $\i_D$ and $L_D$ on $\O^\bullet$ and show that they obey the conditions $\i_D\circ \i_E=\i_{D\cup E}$, $L_{[D,E]}=[L_D,L_E]_\circ$. Then we need to check the identities $[B,\i_D]_\circ =L_D$ and $[L_D,\i_E]_\circ=(-1)^{|D|}\i_{[D,E]}$. The last identity $L_{D\cup E}=L_D\circ \i_E+(-1)^{|D|+1}\i_D\circ L_E$ follows automatically.

From now on we take the decomposition of $x$ as simply $[d]=J_1\dscup J_2...$. The range of summation is sometimes omitted.

\subsection{The module structure} \label{subsec_mod}
\begin{defn}
 For each homogeneous $D\in C^\bullet(\gg;A,A)$ we have an element $\i_D\in\hom^{|D|+1}(C_\bullet(\gg;A,A),C_\bullet(\gg;A,A,))$     \[\iota_D(y_{[a]};x_{[d]})=(-1)^{|D|}\sum(-1)^\ast y_{I_1}\otimes \qq(y_{I_2};x_{J_5}\otimes x_{J_1}\otimes D(y_{I_3};x_{J_2})\otimes x_{J_3})\otimes x_{J_4}\]
    Here the summation is over all possible decomposition $I_1\sqcup I_2\sqcup I_3=[a]$ and $J_1\dscup J_2\dscup J_3\dscup J_4\dscup J_5=[d]$, where $J_1$ is nonempty. The sign is     \[\ast =\e+|x_{J_5}|(|x_{[d]}|-|x_{J_5}|)+|y_{I_3}|(|x_{J_5}|+|x_{J_1}|)+|D|(|y_{I_1}|+|y_{I_2}|+|x_{J_1}|+|x_{J_5}|)+y_{I_1}\]
    which is simply $\k$, using our sign notation. $\i_D$ is a map of shifted degree $|D|+1$.
\end{defn}
Now we prove the following lemma, by classifying all patterns of words.
\begin{lem}
 There is a chain level identity  \[[b,\i_D]_\circ=b\circ \i_D-(-1)^{|D|+1}\i_D\circ b=\i_{\d D}\]
\end{lem}
\begin{proof}
    For readability, we decompose $b=b_O+b_C$ where $b_O(y_{[a]};x_{[d]})=\sum_{J_1\neq \phi} (-1)^\k y_{I_1}\otimes x_{J_1}\otimes \qq(y_{I_2};x_{J_2})\otimes x_{J_3}+\sum_{J_1\neq \phi} (-1)^\k y_{I_1}\otimes \qq(y_{I_2};x_{J_3}\otimes x_{J_1})\otimes x_{J_2}$ and $b_C(y_{[a]};x_{[d]})=\sum (-1)^{\k+1} \ll(y_{I_1})\wedge y_{I_2}\otimes x_{[d]}$. 
    We have
    \begin{align}
        b_O\circ \i_D(y_{[a]};x_{[d]})&=b_O\left(\sum_{J_1\neq \phi} (-1)^{\k+|D|} y_{I_1}\otimes \qq(y_{I_2};x_{J_5}\otimes x_{J_1}\otimes D(y_{I_3};x_{J_2})\otimes x_{J_3})\otimes x_{J_4}\right)\notag \\
        &=\sum_{J_1\neq \phi}(-1)^{\k+1} y_{I_1}\otimes \qq(y_{I_2};x_{J_7}\otimes x_{J_1}\otimes D(y_{I_3};x_{J_2})\otimes x_{J_3})\otimes x_{J_4}\otimes \qq(y_{I_4};x_{J_5})\otimes x_{J_6}\label{i3.1}\\
        &\quad +\sum_{J_1\neq \phi} (-1)^{\k+|D|} y_{I_1}\otimes \qq(y_{I_2};x_{J_6}\otimes \qq(y_{I_3};x_{J_7}\otimes x_{J_1}\otimes D(y_{I_4};x_{J_2})\otimes x_{J_3})\otimes x_{J_4})\otimes x_{J_5}\label{i3.2}
    \end{align}
\begin{align}
    (-1)^{|D|+2}\i_{D}\circ b_O(y_{[a]};x_{[d]})&=\i_D\left(\sum (-1)^{|D|+\k}y_{I_1}\otimes x_{J_1}\otimes  \qq(y_{I_2};x_{J_2})\otimes x_{J_3}+\sum (-1)^{\k+|D|} y_{I_1}\otimes \qq(y_{I_2};x_{J_3}\otimes x_{J_1})\otimes x_{J_2}\right)\notag\\
    &=\sum_{J_1\neq \phi} (-1)^{\k}y_{I_1}\otimes \qq(y_{I_2};x_{J_7}\otimes x_{J_1}\otimes D(y_{I_3};x_{J_2})\otimes x_{J_3})\otimes x_{J_4}\otimes \qq(y_{I_4};x_{J_5})\otimes x_{J_6} \label{i3.3}\\
    &\quad + \sum_{J_1\neq \phi} (-1)^{\k} y_{I_1}\otimes \qq(y_{I_2};x_{J_7}\otimes x_{J_1}\otimes D(y_{I_4};x_{J_2})\otimes x_{J_3}\otimes \qq(y_{I_3};x_{J_4})\otimes x_{J_5})\otimes x_{J_6} \label{i3.4}\\
    &\quad +\sum_{J_1\neq \phi} (-1)^\k y_{I_1}\otimes \qq(y_{I_2};x_{J_7}\otimes x_{J_1}\otimes D(y_{I_3};x_{J_2}\otimes \qq(y_{I_4};x_{J_3})\otimes x_{J_4})\otimes x_{J_5})\otimes x_{J_6}\label{i3.5}\\
    &\quad +\sum_{J_1\neq \phi} (-1)^{\k+|D|}y_{I_1}\otimes \qq(y_{I_2};x_{J_7}\otimes x_{J_1}\otimes \qq(y_{I_3};x_{J_2})\otimes x_{J_3}\otimes D(y_{I_4};x_{J_4})\otimes x_{J_5})\otimes x_{J_6}\label{i3.6}\\
    &\quad +\sum_{J_1\neq \phi} (-1)^{\k+|D|}y_{I_1}\otimes \qq(y_{I_2};x_{J_5}\otimes \qq(y_{I_3};x_{J_6})\otimes x_{J_7}\otimes x_{J_1}\otimes D(y_{I_4};x_{J_2})\otimes x_{J_3})\otimes x_{J_4}\label{i3.7} \\
    &\quad +\sum_{J_1\neq \phi} (-1)^{\k+|D|}y_{I_1}\otimes \qq(y_{I_2};x_{J_6}\otimes \qq(y_{I_3};x_{J_7}\otimes x_{J_1})\otimes x_{J_2}\otimes D(y_{I_4};x_{J_3})\otimes x_{J_4})\otimes x_{J_5}\label{i3.8}
\end{align}
and 
\begin{align}
    [b_C,\i_D]_\circ&=b_C\left(\sum (-1)^{|D|+\k}y_{I_1}\otimes \qq(y_{I_2};x_{J_5}\otimes x_{J_1}\otimes D(y_{I_3};x_{J_2})\otimes x_{J_3})\otimes x_{J_4}\right)\\
    &\quad -(-1)^{|D|+1}\i_D\left(\sum (-1)^\k \ll(y_{I_1})\wedge y_{I_2}\otimes x_{[d]}\right) \notag \\
    &=\sum_{J_1\neq \phi} (-1)^{|D|+1+\k} \ll(y_{I_1})\wedge y_{I_2}\otimes \qq(y_{I_3};x_{J_5}\otimes x_{J_1}\otimes D(y_{I_3};x_{J_2})\otimes x_{J_3})\otimes x_{J_4}\label{i3.10}\\
    &\quad +\sum_{J_1\neq \phi} (-1)^{|D|+\k} \ll(y_{I_1})\wedge y_{I_2}\otimes \qq(y_{I_3};x_{J_5}\otimes x_{J_1}\otimes D(y_{I_3};x_{J_2})\otimes x_{J_3})\otimes x_{J_4}\label{i3.11}\\
    &\quad +\sum_{J_1\neq \phi} (-1)^{|D|+1+\k}y_{I_1}\otimes \qq(\ll(y_{I_2})\wedge y_{I_3};x_{J_5}\otimes x_{J_1}\otimes D(y_{I_4};x_{J_2})\otimes x_{J_3})\otimes x_{J_4}\label{i3.12}\\
    &\quad +\sum_{J_1\neq \phi} (-1)^{\k} y_{I_1}\otimes \qq(y_{I_2};x_{J_5}\otimes x_{J_1}\otimes D(\ll(y_{I_3})\wedge y_{I_4};x_{J_2})\otimes x_{J_3})\otimes x_{J_4} \label{i3.13}
\end{align}
 \eqref{i3.10} and \eqref{i3.11} cancel. Observe that \eqref{i3.1} cancels with \eqref{i3.3}; \eqref{i3.2} combines with the \eqref{i3.4},\eqref{i3.6}, \eqref{i3.7} and \eqref{i3.8} into
\begin{align}
    \sum_{I_1\sqcup I_2\sqcup I_4=[a]} &(-1)^{\star+|D|}y_{I_1}\otimes \qq\{\qq\}(y_{I_2};x_{J_5}\otimes x_{J_1}\otimes D(y_{I_4};x_{J_2})\otimes x_{J_3})\otimes x_{J_4}\label{i3.14}\\
    +\sum_{J_1\neq \phi} &(-1)^{\k+|D|+1}y_{I_1}\otimes \qq(y_{I_2};x_{J_7}\otimes x_{J_1}\otimes \qq(y_{I_3};x_{J_2}\otimes D(y_{I_4};x_{J_3})\otimes x_{J_4})\otimes x_{J_5})\otimes x_{J_6}\label{i3.15}
\end{align}
here $\star=\k(y_{I_1}\wedge  y_{I_2}\otimes x_{J_5}\otimes x_{J_1}\otimes D(y_{I_4};x_{J_2})\otimes x_{J_3\dscup J_4})$; \eqref{i3.14} then cancels \eqref{i3.12}, using the OCHA relation $\qq\{\qq\}=\hat \ll (\qq)$. What remain are \eqref{i3.5}, \eqref{i3.13} and \eqref{i3.15}, which combine into $\i_{\d D}(y_{[a]};x_{[d]})$.    Note that \eqref{i3.13} is $\i_{(-1)^{|D|}\hat \ll (D)}(y_{[a]};x_{[d]})$.
\end{proof}

An immediate consequence is that when $D$ is a Hochschild cocycle, the action $\i_D$ descends to Hochschild homology and only depends on the cohomology class of  $D$. In other words,
\begin{prop}
    There is a well-defined action $HH^\ast(\gg;A,A)$ on $HH_\ast(\gg;A,A)$ given by 
    \[[a]\mapsto [D]\cap [a]:=[\i_Da]\]
\end{prop}
Furthermore, under the associative algebra structure of $HH^\ast(\gg;A,A)$ induced by Yoneda product, this action endows  $HH_\ast(\gg;A,A)$ with an $HH^\ast(\gg;A,A)$-module structure

\begin{prop}
    For $D,E$ Hochschild cocycles, we have 
    \[[D]\cap ([E]\cap -)=[D\cup E]\cap -\]
    in $\hom(C_\bullet(\gg;A,A),C_\bullet (\gg;A,A))$. In other words, for $[D],[E]\in HH^\ast(\gg;A,A)$, we have $\i_D\i_E=\i_{D\cup E}$ in $HH_\ast(\gg;A,A)$.
\end{prop}
\begin{proof}
    We claim that
    \[\i_{E\cup D}\simeq (-1)^{(|D|+1)(|E|+1)}\i_D\i_E\]
 Define the homotopy operator as in \cite{chen2022note}, which is a straightforward generalization from \cite{tamarkintsy2005ring,Tsygan2004},
  \[K(E,D)(y_{[a]};x_{[d]})=\sum(-1)^\k \qq(y_{I_1};x_{J_7}\otimes x_{J_1}\otimes E(y_{I_2};x_{J_2})\otimes x_{J_3}\otimes D(y_{I_3};x_{J_4})\otimes x_{J_5})\otimes x_{J_6}\]
    where the summation is over $J_1\dscup J_2\dscup J_3\dscup J_4\dscup J_5\dscup J_6\dscup J_7=[d]$ and $I_1\sqcup I_2\sqcup I_3=[a]$, with $J_1$ being nonempty. The claim follows from the chain level identity
    \[(-1)^{|D||E|+|D|}\i_D\circ \i_{E}+(-1)^{|E|}\i_{E\cup D}+[K(E,D),b]_\circ+(-1)^{|D|}K(E,\d D)+(-1)^{|D|+|E|}K(\d E,D)=0\]

To this end, we simply act the LHS on $(y_{[a]}\otimes x_{[d]})$ and expand every term. We obtain
\begin{align}
    (-1)^{|D||E|+|D|}\i_D\circ \i_E=&(-1)^{|D|+|E|}\times \notag \\\sum (-&1)^\k \textcolor{blue}{y_{I_1}\otimes \qq(y_{I_2};x_{J_9}\otimes x_{J_1}\otimes \qq(y_{I_3};x_{J_2}\otimes E(y_{I_4};x_{J_3})\otimes x_{J_4}\otimes D(y_{I_5};x_{J_5})\otimes x_{J_6})\otimes x_{J_7})\otimes x_{J_8}}\label{cap3.16}
\end{align}
\begin{align}
    (-1)^{|E|}\i_{E\cup D}=& (-1)^{|E|+|D|}\times \notag \\\sum (-&1)^\k \textcolor{blue}{y_{I_1}\otimes \qq(y_{I_2};x_{J_8}\otimes \qq(y_{I_3};x_{J_9}\otimes x_{J_1}\otimes E(y_{I_4};x_{J_2})\otimes x_{J_3})\otimes x_{J_4}\otimes D(y_{I_5};x_{J_5})\otimes x_{J_6})\otimes x_{J_7}}\label{cap3.17}
\end{align}
The commutator expands as follows 
\begin{align}
    K(E,D)&\circ b(y_{[a]};x_{[d]})=\notag \\ \sum (-&1)^{\k}\textcolor{blue}{y_{I_1}\otimes \qq\left(y_{I_2};x_{J_9}\otimes x_{J_1}\otimes E(y_{I_3};x_{J_2})\otimes x_{J_3}\otimes D(y_{I_4};x_{J_4})\otimes x_{J_5}\otimes \qq(y_{I_5};x_{J_6}\right)\otimes x_{J_7})\otimes x_{J_8}}\label{cap3.18} \\
    +\sum (&-1)^\k \textcolor{blue}{y_{I_1}\otimes \qq\left(y_{I_2};x_{J_9}\otimes x_{J_1}\otimes E(y_{I_3};x_{J_2})\otimes x_{J_3}\otimes D(y_{I_4};x_{J_4}\otimes \qq(y_{I_5};x_{J_5}\right)\otimes x_{J_6})\otimes x_{J_7})\otimes x_{J_8}}\label{cap3.19} \\
    +\sum (&-1)^{\k+|D|}\textcolor{blue}{y_{I_1}\otimes \qq(y_{I_2};x_{J_9}\otimes x_{J_1}\otimes E(y_{I_3};x_{J_2})\otimes x_{J_3}\otimes \qq(y_{I_5};x_{J_4})\otimes x_{J_5}\otimes D(y_{I_4};x_{J_6})\otimes x_{J_7})\otimes x_{J_8}}\label{cap3.20} \\
    +\sum (&-1)^{\k+|D|}\textcolor{blue}{y_{I_1}\otimes \qq(y_{I_2};x_{J_9}\otimes x_{J_1}\otimes E(y_{I_3};x_{J_2}\otimes \qq(y_{I_5};x_{J_3})\otimes x_{J_4})\otimes x_{J_5}\otimes D(y_{I_4};x_{J_6})\otimes x_{J_7})\otimes x_{J_8}}\label{cap3.21}\\
    +\sum (&-1)^{\k+|D|+|E|}\textcolor{blue}{y_{I_1}\otimes \qq(y_{I_2};x_{J_9}\otimes x_{J_1}\otimes \qq(y_{I_5};x_{J_2})\otimes x_{J_3}\otimes E(y_{I_3};x_{J_4})\otimes x_{J_5}\otimes D(y_{I_4};x_{J_6})\otimes x_{J_7})\otimes x_{J_8}}\label{cap3.22}
    \end{align} 
    \begin{align}
        +\sum (&-1)^{\k}\textcolor{red}{y_{I_1}\otimes \qq(y_{I_2};x_{J_9}\otimes x_{J_1}\otimes E(y_{I_3};x_{J_2})\otimes x_{J_3}\otimes D(y_{I_4};x_{J_4})\otimes x_{J_5})\otimes x_{J_6}\otimes \qq(y_{I_5};x_{J_7})\otimes  x_{J_8}}\label{cap3.23}\\
    +\sum (&-1)^{\k+|D|+|E|}\textcolor{blue}{y_{I_1}\otimes \qq(y_{I_2};x_{J_7}\otimes \qq(y_{I_5};x_{J_8})\otimes x_{J_9}\otimes x_{J_1}\otimes E(y_{I_3};x_{J_2})\otimes x_{J_3}\otimes D(y_{I_4};x_{J_4})\otimes x_{J_5})\otimes x_{J_6}}\label{cap3.24}\\
    +\sum (&-1)^{\k+|E|+|D|}\textcolor{blue}{y_{I_1}\otimes \qq(y_{I_2};x_{J_8}\otimes \qq(y_{I_5};x_{J_9}\otimes x_{J_1})\otimes x_{J_2}\otimes E(y_{I_3};x_{J_3})\otimes x_{J_4}\otimes D(y_{I_4};x_{J_5})\otimes x_{J_6})\otimes x_{J_7}}\label{cap3.25}\\
    +\sum (&-1)^{\k+|E|+|D|}\textcolor{blue}{\ll(y_{I_1})\wedge y_{I_5}\otimes \qq(y_{I_2};x_{J_7}\otimes x_{J_1}\otimes E(y_{I_3};x_{J_2})\otimes x_{J_3}\otimes D(y_{I_4};x_{J_4})\otimes x_{J_5})\otimes x_{J_6}}\label{cap3.26}\\
    +\sum (&-1)^{\k+|E|+|D|+1}\textcolor{purple}{y_{I_1}\otimes \qq(\ll(y_{I_2})\wedge y_{I_5};x_{J_7}\otimes x_{J_1}\otimes E(y_{I_3};x_{J_2})\otimes x_{J_3}\otimes D(y_{I_4};x_{J_4})\otimes x_{J_5})\otimes x_{J_6}}\label{cap3.27}\\
    +\sum (&-1)^{\k+|D|+1}\textcolor{blue}{y_{I_1}\otimes \qq(y_{I_2};x_{J_7}\otimes x_{J_1}\otimes E(\ll(y_{I_3})\wedge y_{I_5};x_{J_2})\otimes x_{J_3}\otimes D(y_{I_4};x_{J_4})\otimes x_{J_5})\otimes x_{J_6}}\label{cap3.28}\\
    +\sum (&-1)^{\k+1}\textcolor{blue}{y_{I_1}\otimes \qq(y_{I_2};x_{J_7}\otimes x_{J_1}\otimes E(y_{I_3};x_{J_2})\otimes x_{J_3}\otimes D(\ll(y_{I_4})\wedge y_{I_5};x_{J_4})\otimes x_{J_5})\otimes x_{J_6}}\label{cap3.29}
\end{align}

We observe that \eqref{cap3.27} is nothing but
$$\sum_{I_1\sqcup I_2\sqcup I_3\sqcup I_4=[a]} (-1)^{|D|+|E|+1+\star}\textcolor{purple}{y_{I_1}\otimes\hat \ll(\qq)(y_{I_2};x_{J_7}\otimes x_{J_1}\otimes E(y_{I_3};x_{J_2})\otimes x_{J_3}\otimes D(y_{I_4};x_{J_4})\otimes x_{J_5})\otimes x_{J_6}}$$
Here $\star=\k(y_{I_1}\wedge y_{I_2};x_{J_7}\otimes x_{J_1}\otimes E(y_{I_3};x_{J_2})\otimes x_{J_3}\otimes D(y_{I_4};x_{J_4})\otimes x_{J_5\dscup J_6})$. 

\begin{align}
    -(-1)^{|D|+|E|+1}b&\circ K(E,D)(y_{[a]};x_{[d]})=\notag \\ \sum (&-1)^{\k+|E|+|D|}\textcolor{blue}{y_{I_1}\otimes \qq\left(y_{I_2};x_{J_8}\otimes \qq(y_{I_5};x_{J_9}\otimes x_{J_1}\otimes E(y_{I_3};x_{J_2})\otimes x_{J_3}\otimes D(y_{I_4};x_{J_4})\otimes x_{J_5})\otimes x_{J_6} \right)\otimes x_{J_7}}\label{cap3.30}\\
    +\sum (&-1)^{\k+1}\textcolor{red}{y_{I_1}\otimes \qq(y_{I_2};x_{J_9}\otimes x_{J_1}\otimes E(y_{I_3};x_{J_2})\otimes x_{J_3}\otimes D(y_{I_4};x_{J_4})\otimes x_{J_5})\otimes x_{J_6}\otimes \qq(y_{I_5};x_{J_7})\otimes x_{J_8}}\label{cap3.31}\\
    +\sum (&-1)^{\k+1+|D|+|E|}\textcolor{blue}{\ll(y_{I_1})\wedge y_{I_5}\otimes \qq(y_{I_2};x_{J_7}\otimes x_{J_1}\otimes E(y_{I_3};x_{J_2})\otimes x_{J_3}\otimes D(y_{I_4};x_{J_4})\otimes x_{J_5})\otimes x_{J_6}}\label{cap3.32}
\end{align}

Observe that \eqref{cap3.31} cancels \eqref{cap3.23}. \eqref{cap3.16}-\eqref{cap3.22}, \eqref{cap3.24}-\eqref{cap3.26}, \eqref{cap3.28}-\eqref{cap3.30} and \eqref{cap3.32} together with $K(E,\d D)$ and $K(\d E,D)$ form exactly the expression $$\sum (-1)^{|D|+|E|+\star}y_{I_1}\otimes\qq \{\qq\}(y_{I_2};x_{J_7}\otimes x_{J_1}\otimes E(y_{I_3};x_{J_2})\otimes x_{J_3}\otimes D(y_{I_4};x_{J_4})\otimes x_{J_5})\otimes x_{J_6}$$ For a fixed decomposition, it cancels \eqref{cap3.27} by OCHA relation. Hence the claim is proved.
    
    The statement then follows by graded commutativity of Yoneda product at cohomology level,
    \[[E\cup D]=(-1)^{(|E|+1)(|D|+1)}[D\cup E]\]
    i.e. $[\i_{[E\cup D]}(a)]=[\i_{[E]}\i_{[D]}(a)]$.
\end{proof}

\subsection{The Lie derivatives}\label{subsec_Lie}

It remains to show the last three identities involving the Lie derivative. 
\begin{defn}
    \[L_D(y_{[a]};x_{[d]})=(-1)^{|D|}\sum (-1)^{\k} y_{I_1}\otimes x_{J_1}\otimes D(y_{I_2};x_{J_2})\otimes x_{J_3}+(-1)^{|D|}\sum (-1)^\k y_{I_1}\otimes D(y_{I_2};x_{J_3}\otimes x_{J_1})\otimes x_{J_2}\]
    where as usual, the summation is over decomposition $J_1\dscup J_2\dscup J_3=[d]$, $I_1\sqcup I_2=[a]$, $J_1$ is nonempty. It is clear that $|L_D|=|D|$.
\end{defn}

Observe that the definition does not depend on the OCHA structure. Hence one can check the identity $[L_D,L_E]=L_{[D,E]}$ in the same way as for associative algebras. We omit it here. 
\begin{prop}
   We have 
    \[[B,\i_D]_\circ+[b,S_D]_\circ-S_{\d D}=L_D\]
    Here \[S_D(y_{[a]};x_{[d]})=\sum(-1)^{\k+|D|}y_{I_1}\otimes \one \otimes x_{J_4}\otimes x_{J_1}\otimes D(y_{I_2};x_{J_2})\otimes x_{J_3}\]
    where $J_1\neq \phi$ as always. We see that $|S_D|=|D|-1$.
\end{prop}
\begin{rmk}\label{rmk_cart_hom}
        It is obvious that $[B,S_D]=0$ in the normalized Hochschild chain complex. Hence the proposition can be recast in a more familiar way $[b+uB,\i_D+uS_D]=uL_D+\i_{\d D}+uS_{\d D}$. This is the non-commutative analogue of Cartan homotopy formula \cite{Rinehart1963,Getzler2002CARTANHF}.
\end{rmk}
\begin{proof}
    We compute
    \begin{align}
        B\circ \i_D(y_{[a]};x_{[d]})=\sum(-1)^{\k+|D|}\textcolor{RubineRed}{y_{I_1}\otimes \one \otimes x_{J_5}\otimes \qq(y_{I_2};x_{J_6}\otimes x_{J_1}\otimes D(y_{I_3};x_{J_2})\otimes x_{J_3})\otimes x_{J_4}}\label{L3.33}
    \end{align}
    
    \begin{align}
        -(-1)^{|D|+1}\i_D&\circ  B(y_{[a]};x_{[d]})=\notag \\
         \sum (&-1)^{\k+|D|}y_{I_1}\otimes\qq(y_{I_2};x_{J_2}\otimes \one \otimes x_{J_3}\otimes D(y_{I_3};x_{J_4})\otimes x_{J_5})\otimes x_{J_6}\otimes x_{J_1} \notag  \\
         +\sum (&-1)^{\k+|D|}y_{I_1}\otimes \qq(y_{I_2};x_{J_3}\otimes \one \otimes x_{J_4}\otimes D(y_{I_3};x_{J_5})\otimes x_{J_6}\otimes x_{J_1})\otimes x_{J_2}  \notag  \\ 
         +\sum (&-1)^{\k+|D|} y_{I_1}\otimes \qq(y_{I_2};x_{J_4}\otimes \one \otimes x_{J_5}\otimes D(y_{I_3};x_{J_6}\otimes x_{J_1})\otimes x_{J_2})\otimes x_{J_3}  \notag  \\
         +\sum (&-1)^{\k+|D|}  y_{I_1}\otimes \qq(y_{I_2};x_{J_5}\otimes \one \otimes x_{J_6}\otimes x_{J_1}\otimes D(y_{I_3};x_{J_2})\otimes x_{J_3})\otimes x_{J_4} \notag \\
         +\sum (&-1)^{\k+|D|} y_{I_1}\otimes \qq(y_{I_2};x_{J_6}\otimes x_{J_1}\otimes \one \otimes x_{J_2}\otimes D(y_{I_3};x_{J_3})\otimes x_{J_4})\otimes x_{J_5} \notag 
    \end{align}
By strict unitality we see that the only terms that remain are from $I_2=\phi$, i.e. $\sum (-1)^{\k+|D|}y_{I_1}\otimes \qq_{0,2}(\one,D(y_{I_3};x_{J_2}))\otimes x_{J_3}\otimes x_{J_1}$ and $\sum (-1)^{\k+|D|}y_{I_1}\otimes \qq_{0,2}(\one,D(y_{I_3};x_{J_3}\otimes x_{J_1}))\otimes x_{J_2}$. Hence we obtain 
\begin{align}
    (-1)^{|D|}\i_D&\circ B(y_{[a]};x_{[d]})=\notag \\
    \sum (&-1)^{\k+|D|} \textcolor{blue}{y_{I_1}\otimes D(y_{I_2};x_{J_2})\otimes x_{J_3}\otimes x_{J_1}  }+\sum (-1)^{\k+|D|} \textcolor{ForestGreen}{ y_{I_1}\otimes  D(y_{I_2};x_{J_3}\otimes x_{J_1})\otimes x_{J_2}}\label{L3.34}
\end{align}
where $J_1\neq \phi$. 
Similarly, using the unitality of OCHA one computes
\begin{align}
    b\circ S_D=& \notag \\
    \sum (-&1)^{\k+|D|}y_{I_1}\otimes x_{J_4}\otimes x_{J_1}\otimes D(y_{I_2};x_{J_2})\otimes x_{J_3} \label{L3.35}\\
    +\sum (-&1)^{\k+1+|D|}\textcolor{BurntOrange}{y_{I_1}\otimes \one \otimes x_{J_6}\otimes x_{J_1}\otimes \qq(y_{I_2};x_{J_2})\otimes x_{J_3}\otimes  D(y_{I_3};x_{J_4})\otimes x_{J_5}}\label{L3.36}\\
    +\sum (-&1)^{\k+1+|D|}\textcolor{BurntOrange}{y_{I_1}\otimes \one \otimes x_{J_4}\otimes \qq(y_{I_2};x_{J_5})\otimes x_{J_6}\otimes x_{J_1}\otimes  D(y_{I_3};x_{J_2})\otimes x_{J_3}}\label{L3.37}\\
    +\sum (-&1)^{\k+1+|D|}\textcolor{TealBlue}{y_{I_1}\otimes \one\otimes  x_{J_6}\otimes x_{J_1}\otimes \qq(y_{I_2};x_{J_2}\otimes D(y_{I_3};x_{J_3})\otimes x_{J_4})\otimes x_{J_5}}\label{L3.38}\\
    +\sum (-&1)^{\k+1}\textcolor{BurntOrange}{y_{I_1}\otimes \one\otimes  x_{J_6}\otimes x_{J_1}\otimes D(y_{I_3};x_{J_2})\otimes x_{J_3}\otimes \qq(y_{I_2};x_{J_4})\otimes x_{J_5}}\label{L3.39}\\
    +\sum (-&1)^{\k+1+|D|}\textcolor{BurntOrange}{y_{I_1}\otimes \one \otimes x_{J_5}\otimes \qq(y_{I_2};x_{J_6}\otimes x_{J_1})\otimes x_{J_2}\otimes D(y_{I_3};x_{J_3})\otimes x_{J_4}}\label{L3.40}\\
    +\sum (-&1)^{\k+1+|D|}\textcolor{RubineRed}{y_{I_1}\otimes \one \otimes x_{J_5}\otimes \qq(y_{I_2};x_{J_6}\otimes x_{J_1}\otimes D(y_{I_3};x_{J_2})\otimes x_{J_3})\otimes x_{J_4}}\label{L3.41}\\
    +\sum_{J_1\dscup J_2\dscup J_3\dscup \{j\}\dscup J_4=[d]} (-&1)^{\k+1+|D|} y_{I_1}\otimes x_{j}\otimes x_{J_4}\otimes x_{J_1}\otimes D(y_{I_2};x_{J_2})\otimes x_{J_3} \label{L3.42} \\
    +\sum_{J_1\dscup J_2 \dscup J_4=[d]} (-&1)^{\k+|D|+1} y_{I_1}\otimes D(y_{I_2};x_{J_2})\otimes x_{J_4}\otimes x_{J_1}\label{L3.43}\\
    +\sum (-&1)^{\k+|D|+1|}\textcolor{RoyalPurple}{\ll(y_{I_1})\wedge y_{I_2}\otimes \one \otimes x_{J_4}\otimes x_{J_1}\otimes D(y_{I_3};x_{J_2})\otimes x_{J_3}}\label{L3.44}
\end{align}
\eqref{L3.35},\eqref{L3.42} and \eqref{L3.43} simplify into 
\begin{align}
\sum (-1)^{\k+|D|}\textcolor{ForestGreen}{y_{I_1}\otimes x_{J_1}\otimes D(y_{I_2};x_{J_2})\otimes x_{J_3}}+\sum (-1)^{\k+|D|+1}\textcolor{blue}{y_{I_1}\otimes D(y_{I_2};x_{J_2})\otimes x_{J_3}\otimes x_{J_1}} \label{L3.45}
\end{align}
with $J_1\neq \phi$. The second summation cancels the first summation in \eqref{L3.34}, while the first summation combines with the second summation in \eqref{L3.34} into precisely $L_D$. 

\begin{align}
    -(-1)^{|S_D|}S_D&\circ b= \notag \\
    \sum (&-1)^{\k}\textcolor{BurntOrange}{y_{I_1}\otimes \one \otimes x_{J_6}\otimes x_{J_1}\otimes D(y_{I_3};x_{J_2})\otimes x_{J_3}\otimes \qq(y_{I_2};x_{J_4})\otimes x_{J_6}}\label{L3.46}\\
    +\sum (&-1)^{\k}\textcolor{TealBlue}{ y_{I_1}\otimes \one \otimes x_{J_6}\otimes x_{J_1}\otimes D(y_{I_3};x_{J_2}\otimes \qq(y_{I_2};x_{J_3})\otimes x_{J_4})\otimes x_{J_5} }\label{L3.47}\\
    +\sum (&-1)^{\k+|D|} \textcolor{BurntOrange}{y_{I_1}\otimes \one \otimes x_{J_6}\otimes x_{J_1}\otimes \qq(y_{I_2};x_{J_2})\otimes x_{J_3}\otimes D(y_{I_3};x_{J_4})\otimes x_{J_5}}\label{L3.48}\\
    +\sum (&-1)^{\k+|D|} \textcolor{BurntOrange}{y_{I_1}\otimes \one \otimes x_{J_4}\otimes \qq(y_{I_2};x_{J_5})\otimes x_{J_6}\otimes x_{J_1}\otimes D(y_{I_3};x_{J_2})\otimes x_{J_3}}\label{L3.49}\\
    +\sum (&-1)^{\k+|D|}\textcolor{BurntOrange}{y_{I_1}\otimes x_{J_5}\otimes \qq(y_{I_2};x_{J_6}\otimes x_{J_1})\otimes x_{J_2}\otimes D(y_{I_3};x_{J_3})\otimes x_{J_4}}\label{L3.50}\\
    +\sum(&-1)^{\k+|D|}\textcolor{RoyalPurple}{\ll(y_{I_1})\wedge y_{I_2}\otimes \one \otimes x_{J_4}\otimes x_{J_1}\otimes D(y_{I_3};x_{J_2})\otimes x_{J_3}}\label{L3.51}\\
    +\sum (&-1)^{\k+1}\textcolor{TealBlue}{y_{I_1}\otimes \one \otimes x_{J_4}\otimes x_{J_1}\otimes D(\ll(y_{I_2})\wedge y_{I_3};x_{J_2})\otimes x_{J_3}}\label{L3.52}
\end{align}
    Cross terms \eqref{L3.36}, \eqref{L3.37}, \eqref{L3.39}, \eqref{L3.40},\eqref{L3.46} and \eqref{L3.48}-\eqref{L3.50} cancel. \eqref{L3.44} cancels \eqref{L3.51}. \eqref{L3.33} cancels \eqref{L3.41}. \eqref{L3.47} and \eqref{L3.52} cancel with $S_{\d D}$. We are done.
\end{proof}
Finally
\begin{prop}
    We have 
    \[[L_D,\i_E]_\circ-T(\d D,E)+(-1)^{|D|}T(D,\d E)+[b,T(D,E)]_{\circ}=(-1)^{|D|}\i_{[D,E]}\]
    Here \[T(D,E)(y_{[a]};x_{[d]})=\sum (-1)^{|E|+1+\k}y_{I_1}\otimes D(y_{I_2};x_R\otimes x_L\otimes E(y_{I_3};x_{M_1})\otimes a_{M_2})\otimes a_{M_3}\]
    where the decomposition is $L\dscup M_1\dscup M_2\dscup M_3\dscup R=[d]$, $L\neq \phi$. We see that $|T(D,E)|=|D|+|E|$.
\end{prop}
\begin{proof}
    Expanding, we see that 
    \begin{align}
        L_D&\circ \i_E(y_{[a]};x_{[d]})=\notag \\
         \sum (&-1)^{\k+|D||E|+|E|} \textcolor{NavyBlue}{y_{I_1}\otimes \qq(y_{I_2};x_{J_7}\otimes x_{J_1}\otimes E(y_{I_3};x_{J_2})\otimes x_{J_3})\otimes x_{J_4}\otimes D(y_{I_4};x_{J_5})\otimes x_{J_6}}\label{li3.53} \\
         +\sum (&-1)^{\k+|E|+|D|}\textcolor{RubineRed}{y_{I_1}\otimes D(y_{I_2};x_{J_6}\otimes \qq(y_{I_3};x_{J_7}\otimes x_{J_1}\otimes E(y_{I_4};x_{J_2})\otimes x_{J_3})\otimes x_{J_4})\otimes x_{J_5}}\label{li3.54}
    \end{align}
\begin{align}
    -(-1)^{(|E|+1)|D|} \i_E&\circ L_D(y_{[a]};x_{[d]})= \notag \\
     \sum (&-1)^{\k+|E||D|+|E|+1}\textcolor{NavyBlue}{y_{I_1}\otimes \qq(y_{I_2};x_{J_7}\otimes x_{J_1}\otimes E(y_{I_3};x_{J_2})\otimes x_{J_3})\otimes x_{J_4}\otimes D(y_{I_4};x_{J_5})\otimes x_{J_6} } \label{li3.55} \\
     +\sum (&-1)^{\k+|E||D|+|E|+1}\textcolor{BurntOrange}{y_{I_1}\otimes \qq(y_{I_2};x_{J_7}\otimes x_{J_1}\otimes E(y_{I_3};x_{J_2})\otimes x_{J_3}\otimes D(y_{I_4};x_{J_4})\otimes x_{J_5})\otimes x_{J_6}}\label{li3.56}\\
     +\sum (&-1)^{\k+|E||D|+|E|+1}y_{I_1}\otimes \qq(y_{I_2};x_{J_7}\otimes x_{J_1}\otimes E(y_{I_3};x_{J_2}\otimes D(y_{I_4};x_{J_3})\otimes x_{J_4})\otimes x_{J_5})\otimes x_{J_6}\label{li3.57}\\
     +\sum (&-1)^{\k+|E|+1}\textcolor{BurntOrange}{y_{I_1}\otimes \qq(y_{I_2};x_{J_7}\otimes x_{J_1}\otimes D(y_{I_3};x_{J_2})\otimes x_{J_3}\otimes E(y_{I_4};x_{J_4})\otimes x_{J_5})\otimes x_{J_6}}\label{li3.58}\\
     +\sum (&-1)^{\k+|E|+1}\textcolor{BurntOrange}{y_{I_1}\otimes \qq(y_{I_2};x_{J_5}\otimes D(y_{I_3};x_{J_6})\otimes x_{J_7}\otimes x_{J_1}\otimes  E(y_{I_4};x_{J_2})\otimes x_{J_3})\otimes x_{J_4}}\label{li3.59}\\
     +\sum (&-1)^{\k+|E|+1}\textcolor{BurntOrange}{y_{I_1}\otimes \qq(y_{I_2};x_{J_6}\otimes D(y_{I_3};x_{J_7}\otimes x_{J_1})\otimes x_{J_2}\otimes E(y_{I_4};x_{J_3})\otimes x_{J_4})\otimes x_{J_5}}\label{li3.60}
\end{align}
and  
\begin{align}
    [b,T(D,E)]&_\circ = \notag \\
    \sum (&-1)^{\k+1+|D|}\textcolor{Maroon}{y_{I_1}\otimes D(y_{I_2};x_{J_7}\otimes x_{J_1}\otimes E(y_{I_3};x_{J_2})\otimes x_{J_3})\otimes x_{J_4}\otimes \qq(y_{I_4};x_{J_5})\otimes x_{J_6}}\label{li3.61} \\
    +\sum (&-1)^{\k+|E|+1}\textcolor{BurntOrange}{y_{I_1}\otimes \qq(y_{I_2};x_{J_6}\otimes D(y_{I_3};x_{J_7}\otimes x_{J_1}\otimes E(y_{I_4};x_{J_2})\otimes x_{J_3})\otimes x_{J_4})\otimes x_{J_5}}\label{li3.62}\\
    +\sum (&-1)^{\k+|D|}\textcolor{Maroon}{y_{I_1}\otimes D(y_{I_2};x_{J_7}\otimes x_{J_1}\otimes E(y_{I_3};x_{J_2})\otimes x_{J_3})\otimes x_{J_4}\otimes \qq(y_{I_4};x_{J_5})\otimes x_{J_6}}\label{li3.63}\\
    +\sum (&-1)^{\k+|D|}\textcolor{ForestGreen}{y_{I_1}\otimes D(y_{I_2};x_{J_7}\otimes x_{J_1}\otimes E(y_{I_3};x_{J_2}\otimes \qq(y_{I_4};x_{J_3})\otimes x_{J_4})\otimes x_{J_5})\otimes  x_{J_6}}\label{li3.64}\\
    +\sum (&-1)^{\k+|D|+|E|}\textcolor{RubineRed}{y_{I_1}\otimes D(y_{I_2};x_{J_7}\otimes x_{J_1}\otimes \qq(y_{I_3};x_{J_2})\otimes x_{J_3}\otimes E(y_{I_4};x_{J_4}) \otimes x_{J_5})\otimes x_{J_6}}\label{li3.65}\\
    +\sum (&-1)^{\k+|D|}\textcolor{RubineRed}{y_{I_1}\otimes D(y_{I_2};x_{J_7}\otimes x_{J_1}\otimes E(y_{I_3};x_{J_2})\otimes x_{J_3}\otimes \qq(y_{I_4};x_{J_4})\otimes x_{J_5})\otimes x_{J_6}}\label{li3.66}\\
    +\sum (&-1)^{\k+|D|+|E|}\textcolor{RubineRed}{y_{I_1}\otimes D(y_{I_2};x_{J_5}\otimes \qq(y_{I_3};x_{J_6})\otimes x_{J_7}\otimes x_{J_1}\otimes E(y_{I_4};x_{J_2})\otimes x_{J_3})\otimes x_{J_4}}\label{li3.67}\\
     +\sum (&-1)^{\k+|D|+|E|}\textcolor{RubineRed}{y_{I_1}\otimes D(y_{I_2};x_{J_6}\otimes \qq(y_{I_3};x_{J_7}\otimes x_{J_1})\otimes x_{J_2}\otimes E(y_{I_4};x_{J_3})\otimes x_{J_4})\otimes x_{J_5}}\label{li3.68}\\
     +\sum(&-1)^{|E|+\k} \textcolor{red}{\ll(y_{I_1})\wedge y_{I_2}\otimes D(y_{I_3};x_{J_5}\otimes x_{J_1}\otimes E(y_{I_4};x_{J_2})\otimes x_{J_3})\otimes x_{J_4}}\label{li3.69}\\
     +\sum(&-1)^{1+|E|+\k}\textcolor{red}{\ll(y_{I_1})\wedge y_{I_2}\otimes D(y_{I_3};x_{J_5}\otimes x_{J_1}\otimes E(y_{I_4};x_{J_2})\otimes x_{J_3})\otimes x_{J_4}}\label{li3.70}\\
     +\sum(&-1)^{1+|D|+|E|+\k} \textcolor{red}{y_{I_1}\otimes D(\ll (y_{I_2})\wedge y_{I_3};x_{J_5}\otimes x_{J_1}\otimes E(y_{I_4};x_{J_2})\otimes x_{J_3})\otimes x_{J_4}}\label{li3.71}\\
     +\sum(&-1)^{1+|D|+\k}\textcolor{red}{\ll (y_{I_1})\otimes D(y_{I_2};x_{J_5}\otimes x_{J_1}\otimes E(\ll(y_{I_3})\wedge y_{I_4};x_{J_2})\otimes x_{J_3})\otimes x_{J_4}}\label{li3.72}
\end{align}
    
    For completeness we expand the homotopy operators as well
    \begin{align}
        -T(\d D,E)&(y_{[a]};x_{[d]})= \notag \\
        \sum &(-1)^{|E|+\star_1}\textcolor{Plum}{y_{I_1}\otimes \qq\{D\}(y_{I_2};x_{J_5}\otimes x_{J_1}\otimes E(y_{I_3};x_{J_2})\otimes x_{J_3})\otimes x_{J_4}}\label{li3.73}\\
        +\sum &(-1)^{|D|+|E|+\star_1+1}\textcolor{Cerulean}{y_{I_1}\otimes D\{\qq\}(y_{I_2};x_{J_5}\otimes x_{J_1}\otimes E(y_{I_3};x_{J_2})\otimes x_{J_3})\otimes x_{J_4}}\label{li3.74}\\
         +\sum &(-1)^{|E|+|D|+\star_1}\textcolor{Red}{y_{I_1}\otimes \hat \ll(D)(y_{I_2};x_{J_5}\otimes x_{J_1}\otimes E(y_{I_3};x_{J_2})\otimes x_{J_3})\otimes x_{J_4}}\label{li3.75}
    \end{align}
    Here $\star_1=\k(y_{I_1}\wedge y_{I_2}\otimes x_{J_5}\otimes x_{J_1}\otimes E(y_{I_3};x_{J_2})\otimes x_{J_3\dscup J_4})+(1+|D|)|y_{I_1}|$,  
    \begin{align}
        (-1)^{|D|}T(D,\d E)&(y_{[a]};x_{[d]})=\notag \\
        \sum &(-1)^{|D|+|E|+\k} \textcolor{RubineRed}{y_{I_1}\otimes D(y_{I_2};x_{J_7}\otimes x_{J_1}\otimes \qq(y_{I_3};x_{J_2}\otimes E(y_{I_4};x_{J_3})\otimes x_{J_4})\otimes x_{J_5})\otimes x_{J_6} }\label{li3.76}\\
        +\sum &(-1)^{|D|+\k+1} \textcolor{ForestGreen}{ y_{I_1}\otimes D(y_{I_2};x_{J_7}\otimes x_{J_1}\otimes E(y_{I_3};x_{J_2}\otimes \qq(y_{I_4};x_{J_3})\otimes x_{J_4})\otimes x_{J_5})\otimes x_{J_6}} \label{li3.77}\\
         +\sum &(-1)^{|D|+\k}\textcolor{red}{y_{I_1}\otimes D(y_{I_2};x_{J_5}\otimes x_{J_1}\otimes E(\ll(y_{I_3})\wedge y_{I_4};x_{J_2})\otimes x_{J_3})\otimes x_{J_4}}\label{li3.78}
    \end{align}
    \eqref{li3.53} cancels \eqref{li3.55}. \eqref{li3.64} cancels \eqref{li3.77}. \eqref{li3.61} cancels \eqref{li3.63}. \eqref{li3.54}, \eqref{li3.65}-\eqref{li3.68} and \eqref{li3.76} combine into
    \[\sum (-1)^{\star_1+|D|+|E|}\textcolor{Cerulean}{y_{I_1}\otimes D\{\qq\}(y_{I_2};x_{J_5}\otimes x_{J_1}\otimes E(y_{I_3};x_{J_2})\otimes x_{J_3})\otimes x_{J_4}}\]
    which cancels \eqref{li3.74}. \eqref{li3.56},\eqref{li3.58}-\eqref{li3.60} and \eqref{li3.62} combine into 
    \begin{align}
        \sum &(-1)^{|E|+1+\star_1}\textcolor{Plum}{y_{I_1}\otimes \qq\{D\}(y_{I_2};x_{J_5}\otimes x_{J_1}\otimes E(y_{I_3};x_{J_2})\otimes x_{J_3})\otimes x_{J_4}}\label{li3.79}\\
        +\sum &(-1)^{|E|+\k} y_{I_1}\otimes \qq(y_{I_2};x_{J_7}\otimes x_{J_1}\otimes D(y_{I_3};x_{J_2}\otimes E(y_{I_4};x_{J_3})\otimes x_{J_4})\otimes x_{J_5})\otimes x_{J_6}\label{li3.80}
    \end{align}
    \eqref{li3.79} then cancels \eqref{li3.73}. Terms involving $\ll$ are \eqref{li3.69}-\eqref{li3.72},\eqref{li3.75} and \eqref{li3.78}, which cancel. We finally obtain 
\begin{align}
     \sum &(-1)^{\k+|E|} y_{I_1}\otimes \qq(y_{I_2};x_{J_7}\otimes x_{J_1}\otimes D(y_{I_3};x_{J_2}\otimes E(y_{I_4};x_{J_3})\otimes x_{J_4})\otimes x_{J_4})\otimes x_{J_5}\notag \\
     +\sum & (-1)^{\k+|E||D|+|E|+1}y_{I_1}\otimes \qq(y_{I_2};x_{J_7}\otimes x_{J_1}\otimes E(y_{I_3};x_{J_2}\otimes D(y_{I_4};x_{J_3})\otimes x_{J_4})\otimes x_{J_5})\otimes x_{J_6}\notag 
\end{align}
which is $(-1)^{|D|}\i_{[D,E]}$. 
\end{proof}

We summarize the results in this section as follows. Let $(C_\bullet(\gg;A,A),b)$ be the normalized open-closed Hochschild chain complex, $(C^\bullet(\gg;A,A),\d)$ be the normalized open-closed Hochschild cochain complex, both associated to an OCHA $(\gg,A;\ll,\qq)$. For every open-closed cochain $D$ of homogeneous degree $|D|$, one can define a contraction operator $\i_D:C_\bullet(\gg;A,A)\to C_{\bullet+|D|+1} (\gg;A,A)$, a Lie derivative $L_D:C_\bullet(\gg;A,A)\to C_{\bullet+|D|} (\gg;A,A)$. Then
\begin{thm} \label{thm}
     When $D,E$ are Hochschild cocycles, the operators $B,\i_D$ and $L_D$ satisfy the Cartan-type identities up to explicit chain-level homotopies
\[[b,\i_D]_\circ=\i_{\d D},\quad \i_{E\cup D}\simeq (-1)^{(|D|+1)(|E|+1)}\i_D\i_E, \]
\[\quad [B,\i_D]_\circ\simeq L_D,\quad [L_D,L_E]_\circ=L_{[D,E]},\quad  [L_D,\i_E]_\circ \simeq (-1)^{|D|}\i_{[D,E]} \]
\end{thm}
    Here $\simeq$ means equal up to chain homotopy $[b,\cO]_\circ$ for some multi-linear operator $\cO$ on $C_\bullet(\gg;A,A)$. These chain level identities immediately imply the following

\begin{cor}\label{cor_hom}
 There is a calculus structure for an OCHA $(\gg,A)$ with the following data
 \begin{itemize}
        \item  $\mathcal V^\bullet\coloneqq HH^\bullet(\gg;A,A)\coloneqq H^\bullet((C^\bullet(\gg;A,A),\d))$,
        \item  $\O^\bullet \coloneqq HH_\bullet(\gg;A,A)\coloneqq H^\bullet((C_\bullet(\gg;A,A),b)$,
        \item  $B,\i_{[D]},L_{[D]}$ are induced from the chain level maps that appear in Theorem \ref{thm}.
    \end{itemize}
\end{cor}


\section{Getzler-Gauss-Manin connection}\label{sec_ggm}
In this section we show the cyclic-triviality of Getzler-Gauss-Manin connections on a family of OCHAs, as an application of the calculus structure we obtained in the previous section. Now we take $k=\C$. To keep notations comparable with literature, we introduce the operation on $C^\bullet(\gg;A,A)$ as $D\bullet E=D\{E\}$. We use filled circle to distinguish it with the ordinary composition. The proof in this section will be rather short, as it is basically an exercise following \cite{Getzler2002CARTANHF}, using our sign rules.
 
Let $(\gg,A)$ be an OCHA. In particular, $A$ is a unital $A_\infty$-algebra under the structural morphism $\qq_{0,-}$. Recall that a unital OCHA structure is an element $\qq=\qq_1+\qq_2+...\in C_{un}^1(\gg;A,A)$ such that $\qq\{\qq\}=\hat \ll (\qq)$. Following the convention, we make the following definitions of cyclic chain complex and its variants
\begin{defn}
        \begin{align}
        CC_\bullet ^-(\gg;A,A)&\coloneqq (C_\bullet(\gg;A,A)[[u]],b+uB) \\
        CC_\bullet ^{per}(\gg;A,A)&\coloneqq (C_\bullet(\gg;A,A)((u)),b+uB)\\
        CC_\bullet(\gg;A,A) &\coloneqq (C_\bullet(\gg;A,A)((u))/uC_\bullet(\gg;A,A)[[u]],b+uB)
    \end{align}
    Here $u$ is a formal variable of degree 2.
\end{defn}
We will use the periodic cyclic homology so as to divide the formal variable $u$. Now, consider a smooth parameter space $\cM$. Let $\cE^\bullet$ be the free graded $\cO_\cM$-module sheaf $C_\bullet(\gg;A,A)((u))\otimes _\C \cO_\cM$, endowed with shifted degree. We take $\cM=\mathrm{Spf}(\C[[t_1,...,t_n]])$, and consider a family of OCHA structures with a fixed bulk $L_\infty$-structure over $\cM$: this is an OCHA structure of formal power series $\qq_{m,n}(t)\in C^\bullet(\gg;A,A)[[t]]$. We put $\qq(t)=\sum \qq_{m,n}(t)$. We put $b(t)$ for the Hochschild differential varying in family. Now $b+uB$ gives a section of $End_{\cO_\cM\mathrm{-mod}}^1(\cE^\bullet)$ that squares to zero. This datum allows us to define the cohomology sheaf $\cH^i(\cE^\bullet,b+uB)$ of $\cE^\bullet$. 

Recall the Cartan homotopy formula (Remark \ref{rmk_cart_hom}). Let $\i_D+uS_D=R_D$ for $D\in \Gamma(End^{|D|}(\cE^\bullet))$. We make the following definition
\begin{defn}[Proposition-Definition]
    There is an \textit{open-closed Getzler-Gauss-Manin connection} on $\cM$
    \begin{align}
         \nabla\coloneqq d-u^{-1}\sum_i R_{\p_i\qq_t}dt_i:\cE^\bullet\to \cE^\bullet\otimes \O_\cM^1
    \end{align}
    which descends to $\cH^\bullet(\cE)$.
\end{defn}
\begin{proof}
    The Leibnitz rule is clear. It suffices to check     \[[\p_i-u^{-1}\i_{\p_i\qq_t}-S_{\p_i\qq_t},b_t+uB]_\circ=0\]
    where $b_t+uB$ is regarded as a section of $End_{\cO_\cM\mathrm{-mod}}^1(\cE^\bullet)$. $\p_i$ is the conventional basis of tangent vectors. This follows from the following observation. As $B$ does not depend on the OCHA structure, it furnishes a constant section of $\cE^\bullet$ under the trivial connection $d$, so $[\p_i, b_t]_\circ=-L_{\p_i \qq(t)}$. The Cartan homotopy formula, Remark \ref{rmk_cart_hom}, implies that the only other contribution to the LHS is $L_{\p_i \qq_t}$. Indeed, taking derivatives of the OCHA relation $\qq(t)\{\qq(t)\}=\qq(t)\{\ll\}$ gives $[\qq,\p_i\qq(t)]=\hat \ll(\p_i \qq(t))$. In other words, $\d (\p_i \qq(t))=0$. 
\end{proof}
The following theorem is a generalization of \cite{Getzler2002CARTANHF}. We define the operator $\s(D,E)=R_{\{D,E\}}+(-1)^{|D||E|}R_{\{E,D\}}+R_{D\bullet E}$.
\begin{thm}[compare\cite{Getzler2002CARTANHF}, Theorem 3.3]\label{thm_GGM}
   The curvature of open-closed Getzler-Gauss-Manin connection is given by 
   \[[\nabla_{\p_i},\nabla_{\p_j}]=[b+uB,u^{-2}\s(\p_i \qq_t,\p_j\qq_t)-u^{-1}R_{\p_i\qq_t}\circ L_{\p_j\qq_t}]_\circ\]
\end{thm}
\begin{proof}
    We define the following operators
    \begin{align}
        \iota_{\{D,E\}}(y_{[a]};x_{[d]})&\coloneqq\sum(-1)^{\k+|D|+|E|} y_{I_1}\otimes \qq(y_{I_2};x_{J_7}\otimes x_{J_1}\otimes D(y_{I_3};x_{J_2})\otimes x_{J_3}\otimes E(y_{I_4};x_{J_4})\otimes x_{J_5})\otimes x_{J_6}\\
        S_{\{D,E\}}(y_{[a]};x_{[d]})&\coloneqq \sum(-1)^{\k+|D|+|E|}y_{I_1}\otimes \one \otimes x_{J_6}\otimes x_{J_1}\otimes D(y_{I_2};x_{J_2})\otimes x_{J_3}\otimes E(y_{I_3};x_{J_4})\otimes x_{J_5}\\
    L_{\{D,E\}}&\coloneqq\sum (-1)^{\k+|D|+|E|} y_{I_1}\otimes x_{J_1}\otimes D(y_{I_2};x_{J_2})\otimes x_{J_3}\otimes E(y_{I_3};x_{J_4})\otimes x_{J_5} \\
    &+\sum (-1)^{\k+|D||E|+|D|+|E|}y_{I_1}\otimes E(y_{I_2};x_{J_5}\otimes x_{J_1}\otimes D(y_{I_3};x_{J_2})\otimes x_{J_3})\otimes x_{J_4}\notag \\
    &+\sum (-1)^{\k+|D||E|+|D|+|E|}y_{I_1}\otimes E(y_{I_2};x_{J_6}\otimes x_{J_1})\otimes x_{J_2}\otimes D(y_{I_3};x_{J_3})\otimes x_{J_4})\otimes x_{J_5} \notag\\
    \r_{\{D,E\}}&\coloneqq \sum (-1)^{|D|+|E|+\k}y_{I_1}\otimes D(y_{I_2};x_{J_5}\otimes x_{J_1}\otimes E(y_{I_3};x_{J_2})\otimes x_{J_3})\otimes x_{J_4}
\end{align}
Direct computations give the following
\begin{lem}[compare \cite{Getzler2002CARTANHF}, Lemma 2.3]
    \[L_D\circ L_E+\r_{\{D,E\}}+(-1)^{|D||E|}\r_{\{E,D\}}=L_{D\bullet E}+L_{\{D,E\}}+(-1)^{|D||E|}L_{\{E,D\}}\]
\end{lem}
Recall that $\d D=\qq\bullet D-(-1)^{|D|}D\bullet \qq+(-1)^{|D|}\hat \ll(D)$. 
\begin{lem}[compare \cite{Getzler2002CARTANHF}, Lemma 1.4] \label{lem_bullet}
    \[(\d D) \bullet E-\d(D \bullet E)+(-1)^{|D|}D\bullet (\d E)=\qq\{D,E\}+(-1)^{|D||E|}\qq\{E,D\}\]
\end{lem}
Define $\i_{\{D,E\}}+uS_{\{D,E\}}=R_{\{D,E\}}$, we have
\begin{lem}\label{lem_RL}
  \[[b+uB,R_{\{D,E\}}]_\circ- R_{\{\d D,E\}}-(-1)^{|D|}R_{\{D,\d E\}}+(-1)^{|D||E|+|E|}R_E\circ R_D-R_{\qq\{D,E\}}=uL_{\{D,E\}}\]
\end{lem}
\begin{proof}
    It suffices to show that 
\[[B,\i_{\{D,E\}}]_\circ+[b,S_{\{D,E\}}]_\circ+(-1)^{|D||E|+|E|}(S_{E}\circ \i_{D}+\i_E\circ S_D)+S_{\qq\{D,E\}}-S_{\{\d D, E\}}-(-1)^{|D|}S_{\{D,\d E\}}=L_{\{D,E\}}\]
and 
\[[b,\i_{\{D,E\}}]_\circ+(-1)\i_{\{\d D,E\}}+(-1)^{|D|+1}\i_{\{D,\d E\}}+(-1)\i_{\qq\{D,E\}}+(-1)^{|E||D|+|E|}\i_{E}\circ \i_D=0\]
The first equality is true tautologically, though we omit details. The second equality is verified using the OCHA relation. The statement follows, since clearly $[B,S]_\circ=0$, $S\circ S=0$.
\end{proof}
The curvature of Getzler-Gauss-Manin connection is given by 
\begin{align}
    [\p_{i}-u^{-1}R_{\p_i \qq_t},\p_j-u^{-1}R_{\p_i \qq_t}]_\circ=-u^{-1}R_{\p_i\p_j \qq_t}+u^{-1}R_{\p_j\p_i \qq_t}+u^{-1}\r_{\{\p_i\qq_t,\p_j\qq_t\}}-u^{-1}\r_{\{\p_j\qq_t,\p_i\qq_t\}}+u^{-2}[R_{\p_i\qq_t},R_{\p_j\qq_t}]_\circ \label{curv4.9}
\end{align}
    The first two terms on the RHS cancel. We have used that $\p_i R_{D_t}=R_{\p_i D_t}-u^{-1}\r_{\{\p_i\qq_t,D\}}$. 
    
    From Lemma \ref{lem_RL} we already see that 
\[[b+uB,R_{\{D,E\}}]_\circ- R_{\{\d D,E\}}-(-1)^{|D|}R_{\{D,\d E\}}+(-1)^{|D||E|+|E|}R_E\circ R_D-R_{\qq\{D,E\}}=uL_{\{D,E\}}\]
Similarly,
\[[b+uB,R_{\{E,D\}}]_\circ- R_{\{\d E,D\}}-(-1)^{|E|}R_{\{E,\d D\}}+(-1)^{|D||E|+|D|}R_D\circ R_E-R_{\qq\{E,D\}}=uL_{\{E,D\}}\]
Hence let $\s_0(D,E)=R_{\{D,E\}}+(-1)^{|E||D|}R_{\{E,D\}}$ we have 
\begin{align}
    [b+uB,\s_0(D,E)]_\circ-\s_0(\d D,E)-(-1)^{|D|}\s_0(D,\d E)+&(-1)^{|D|}[R_D,R_E]_\circ-R_{\qq\{D,E\}}-(-1)^{|E||D|}R_{\qq\{E,D\}}\notag \\
     &=uL_{\{D,E\}}+u(-1)^{|D||E|}L_{\{E,D\}}\label{s_04.9}
\end{align}
By Remark \ref{rmk_cart_hom} and Lemma \ref{lem_bullet}, 

\begin{align}
    [b+uB,R_{D\bullet E}]_\circ-R_{(\d  D)\bullet E}-(-1)^{|D|}R_{D\bullet (\d E)}+R_{\qq\{D,E\}}+(-1)^{|D||E|}R_{\qq\{E,D\}}=uL_{D\bullet E}\label{RL4.10}
\end{align}
Put $\s(D,E)=R_{\{D,E\}}+(-1)^{|D||E|}R_{\{E,D\}}+R_{D\bullet E}$. Combining \eqref{s_04.9} and \eqref{RL4.10} leads to the following
\begin{align}
    [b+uB,\s(D,E)]_\circ-\s(\d D,E)-(-1)^{|D|}\s(D,\d E)+(-1)^{|D|}[R_D,R_E]_\circ=u\left(L_D\circ L_E+\r_{\{D,E\}}+(-1)^{|D||E|}\r_{\{E,D\}}\right) \label{sigma4.11}
\end{align}
In particular, let $D=\p_i \qq_t$, $E=\p_j \qq_t$ which are closed under the open-closed Hochschild differential $\d$. Substituting \eqref{sigma4.11}, the RHS of \eqref{curv4.9} equals 
\[u^{-2}[b+uB,\s(\p_i \qq_t,\p_j\qq_t)]_\circ-u^{-1}L_{\p_i \qq_t}\circ L_{\p_j \qq_t}\]
Now using the fact that $L_{\p_i \qq_t}\circ L_{\p_j \qq_t}=[b+uB,R_{\p_i\qq_t}\circ L_{\p_j\qq_t}]_\circ$, the proof is complete.
\end{proof}
\begin{rmk}
    The $L_{\{D,E\}}$ operator in the proof is different from the one defined in \cite{Getzler2002CARTANHF}. The reason is basically conventional. In the operators $S$ and $\i$ we always put $x_0$ in front of $D$ and $E$, while Getzler puts $x_0$ behind $D$ and $E$. However as all statements involve only $L_{\{D,E\}}+(-1)^{|D||E|}L_{\{E,D\}}$, this distinction is harmless even in comparing with the literature.
\end{rmk}

	\bibliographystyle{alpha}
	\addtocontents{toc}{\protect\setcounter{tocdepth}{0}}
	\bibliography{myref}		

@article{keller2001introductionainfinityalgebrasmodules,
      title={Introduction to {$A_\infty$} algebras and modules}, 
      author={Bernhard Keller},
      year={2001},
      eprint={math/9910179},
      archivePrefix={arXiv},
      primaryClass={math.RA},
      url={https://arxiv.org/abs/math/9910179}, 
}

@article{Fukaya_toric1,
   title={Lagrangian {Floer} theory on compact toric manifolds, I},
   volume={151},
   ISSN={0012-7094},
   url={http://dx.doi.org/10.1215/00127094-2009-062},
   DOI={10.1215/00127094-2009-062},
   number={1},
   journal={Duke Mathematical Journal},
   publisher={Duke University Press},
   author={Fukaya, Kenji and Oh, Yong-Geun and Ohta, Hiroshi and Ono, Kaoru},
   year={2010},
   month=jan }

@article{fukaya_toric2,
      title={Lagrangian {Floer} theory on compact toric manifolds II : {Bulk} deformations}, 
      author={Kenji Fukaya and Yong-Geun Oh and Hiroshi Ohta and Kaoru Ono},
      year={2011},
      eprint={0810.5654},
      archivePrefix={arXiv},
      primaryClass={math.SG},
      url={https://arxiv.org/abs/0810.5654}, 
}

@article{fukaya_ms_compacttoric,
      title={Lagrangian {Floer} theory and mirror symmetry on compact toric manifolds}, 
      author={Kenji Fukaya and Yong-Geun Oh and Hiroshi Ohta and Kaoru Ono},
      year={2016},
      eprint={1009.1648},
      archivePrefix={arXiv},
      primaryClass={math.SG},
      url={https://arxiv.org/abs/1009.1648}, 
}

@article{HochschildKostantRosenberg1962,
  title={Differential Forms on Regular Affine Algebras},
  author={Hochschild, G. and Kostant, Bertram and Rosenberg, Alex},
  journal={Transactions of the American Mathematical Society},
  volume={102},
  number={3},
  pages={383-408},
  year={1962},
  publisher={American Mathematical Society},
  doi={10.1090/S0002-9947-1962-0142598-8},
  url={https://doi.org}
}

@book{FOOO,
title={Lagrangian intersection {Floer} theory: anomaly and obstruction I and II},
author={Kenji Fukaya and Yong-Geun Oh and Hiroshi Ohta and Kaoru Ono},
series={AMS/IP Studies in Advanced Mathematics},
publisher={Amer. Math. Soc. and International Press},
volume={46},
year={2009},
}

@article{wang2025openclosedhochschildhomologyrelative,
      title={Open-Closed {Hochschild} Homology and the Relative Disk Mapping Space}, 
      author={Yi Wang and Hang Yuan},
      year={2025},
      eprint={2511.05010},
      archivePrefix={arXiv},
      primaryClass={math.AT},
      url={https://arxiv.org/abs/2511.05010}, 
}

@article{yuan2025cyclicbracerelationbv,
      title={Cyclic brace relation and {BV} structure on open-closed Hochschild cohomology}, 
      author={Hang Yuan},
      year={2025},
      eprint={2511.04095},
      archivePrefix={arXiv},
      primaryClass={math.QA},
      url={https://arxiv.org/abs/2511.04095}, 
}

@article{chen2022note,
  author    = {Chen, Youming and Lyu, Weiguo and Yang, Song},
  title     = {A note on the differential calculus of {Hochschild} theory for $A_{\infty}$-algebras},
  journal   = {Electronic Research Archive},
  volume    = {30},
  number    = {9},
  pages     = {3211-3237},
  year      = {2022},
  doi       = {10.3934/era.2022163},
  url       = {https://doi.org}
}

@article{Kajiura_2006homalg,
   title={Homotopy Algebras Inspired by Classical Open-Closed String Field Theory},
   volume={263},
   ISSN={1432-0916},
   url={http://dx.doi.org/10.1007/s00220-006-1539-2},
   DOI={10.1007/s00220-006-1539-2},
   number={3},
   journal={Communications in Mathematical Physics},
   publisher={Springer Science and Business Media LLC},
   author={Kajiura, Hiroshige and Stasheff, Jim},
   year={2006},
   month=Mar, pages={553-581} }

@inproceedings{Getzler2002CARTANHF,
  title={CARTAN HOMOTOPY FORMULAS AND THE {Gauss-Manin} CONNECTION IN CYCLIC HOMOLOGY},
  author={Ezra Getzler},
  year={2002},
  url={https://api.semanticscholar.org/CorpusID:32085481}
}

@article{Rinehart1963,
  title = {Differential forms on general commutative algebras},
  author = {Rinehart, George S.},
  journal = {Transactions of the American Mathematical Society},
  volume = {108},
  number = {2},
  pages = {195-222},
  year = {1963},
  publisher = {American Mathematical Society},
  doi = {10.1090/S0002-9947-1963-0154331-5},
  url = {https://ams.org}
}

@book{loday1992cyclic,
  title={Cyclic Homology},
  author={Loday, J.L.},
  isbn={9783540533399},
  lccn={92034146},
  series={Die Grundlehren der mathematischen Wissenschaften in Einzeldarstellungen},
  url={https://books.google.com.hk/books?id=gFHvAAAAMAAJ},
  year={1992},
  publisher={Springer-Verlag}
}

@article{Sheridan_2019,
   title={Formulae in noncommutative {Hodge} theory},
   volume={15},
   ISSN={1512-2891},
   url={http://dx.doi.org/10.1007/s40062-019-00251-2},
   DOI={10.1007/s40062-019-00251-2},
   number={1},
   journal={Journal of Homotopy and Related Structures},
   publisher={Springer Science and Business Media LLC},
   author={Sheridan, Nick},
   year={2019},
   month=Nov, pages={249-299} }

@article{GelfandDaletskiiTsygan1990,
  author     = {Gelfand, Israel M. and Daletskii, Yuri L. and Tsygan, Boris L.},
  title     = {On a Variant of Noncommutative Differential Geometry},
  journal   = {Soviet Mathematics - Doklady},
  volume    = {40},
  number    = {2},
  pages     = {422-426},
  year      = {1990},
  publisher = {American Mathematical Society}
}

@article{Kajiura_2006ocha,
   title={Open-closed homotopy algebra in mathematical physics},
   volume={47},
   ISSN={1089-7658},
   url={http://dx.doi.org/10.1063/1.2171524},
   DOI={10.1063/1.2171524},
   number={2},
   journal={Journal of Mathematical Physics},
   publisher={AIP Publishing},
   author={Kajiura, Hiroshige and Stasheff, Jim},
   year={2006},
   month=Feb }

@Inbook{Tsygan2004,
author="Tsygan, Boris",
title="Cyclic Homology",
bookTitle="Cyclic Homology in Non-Commutative Geometry",
year="2004",
publisher="Springer Berlin Heidelberg",
address="Berlin, Heidelberg",
pages="73-113",
isbn="978-3-662-06444-3",
doi="10.1007/978-3-662-06444-3_2",
url="https://doi.org/10.1007/978-3-662-06444-3_2"
}

@incollection{tamarkintsy2005ring,
  author    = {Tamarkin, Dmitri and Tsygan, Boris},
  title     = {The ring of differential operators on forms in noncommutative calculus},
  booktitle = {Graphs and patterns in mathematics and theoretical physics},
  series    = {Proceedings of Symposia in Pure Mathematics},
  volume    = {73},
  pages     = {105-131},
  publisher = {American Mathematical Society},
  year      = {2005},
  doi       = {10.1090/pspum/073/2131013}
}

@inproceedings{Weibel1960AnIT,
  title={An Introduction to Homological Algebra: References},
  author={Charles A. Weibel},
  year={1960},
  url={https://api.semanticscholar.org/CorpusID:116651935}
}

@article{yuan2024openclosedstringanaloguehochschild,
      title={An open-closed string analogue of {Hochschild} cohomology}, 
      author={Hang Yuan},
      year={2024},
      eprint={2410.20888},
      archivePrefix={arXiv},
      primaryClass={math.QA},
      url={https://arxiv.org/abs/2410.20888}, 
}
	\addtocontents{toc}{\protect\setcounter{tocdepth}{2}}
	
\end{document}